\documentclass[11pt]{article}%
\usepackage{graphicx}
\usepackage{subcaption}
\usepackage{color}
\usepackage{float} 
\usepackage{fullpage}
\usepackage{adjustbox}
\usepackage{amsfonts}
\usepackage{amsmath}
\usepackage{mathtools}

\usepackage{amssymb}
\usepackage{amsbsy}
\usepackage{amsthm}
\usepackage{epsfig}
\usepackage{graphicx}
\usepackage{colordvi}
\usepackage{graphics}
\usepackage{color}
\usepackage{bm}
\usepackage{algorithm}
\usepackage{algorithmic}
\usepackage{verbatim} 
\usepackage{braket}
\usepackage{authblk}

\numberwithin{equation}{section}

\newcommand{\SB}[1]{\textcolor{black}{#1}}

\newtheorem{theorem}{Theorem}[section]
\newtheorem{thm}{Theorem}[section]

\newtheorem{lemma}{Lemma}[section]

\newtheorem{alg}[thm]{Algorithm}

\newcommand{\nn}{\boldsymbol{n}}

\newcommand{\tu}{\widetilde{u}}
\newcommand{\aaa}{\boldsymbol{a}}

\newcommand{\XTO}{X_{\mathcal{T}}}
\newcommand{\XTT}{X_{\mathcal{T}_-}}
\newcommand{\XOM}{X_{\Omega}}
\newcommand{\XOMM}{X_{\Omega_-}}

\newcommand{\jump}[1]{[\![#1]\!]}
\begin{document}

\title{A Fat boundary-type method for localized nonhomogeneous material problems}

\author{
Alex Viguerie\thanks{Department of Civil Engineering and Architecture, University of Pavia, Pavia, PV 27100, Italy(alexander.viguerie@unipv.it).}\and
Ferdinando Auricchio\thanks{Department of Civil Engineering and Architecture, University of Pavia, Pavia, PV 27100, Italy (auricchio@unipv.it).}\and
Silvia Bertoluzza 
\thanks{CNR Imati Enrico Magenes, Pavia, PV 27100
, Italy (silvia.bertoluzza@imati.cnr.it)}
}

\date{}

\maketitle

\begin{abstract}
Problems with localized nonhomogeneous material properties arise frequently in many applications and are a well-known source of difficulty in numerical simulation\textcolor{black}{s}. In certain applications \textcolor{black}{(including additive manufacturing)}, the physics of the problem may be considerably more complicated in relatively small portions of the domain, requiring a significantly finer \textcolor{black}{local} mesh compared to elsewhere in the domain. This can make the use of a uniform mesh numerically unfeasible. While nonuniform meshes can be employed, they may be challenging to generate (particularly for regions with complex boundaries) and more difficult to precondition. The problem becomes even more prohibitive when the region requiring a finer-level mesh changes in time, requiring the introduction of refinement and derefinement techniques. To address \textcolor{black}{the aforementioned challenges, we employ a technique related to the Fat boundary method \cite{ BMM2005, BIM2011,Maury2001} as a possible alternative.} We analyze the \textcolor{black}{proposed methodology} from a mathematical point of view and validate our findings on two-dimensional numerical tests.
\end{abstract}


\section{Introduction}
\par Problems with localized nonhomogeneous material properties are of great interest in engineering \cite{BRDP2018, CBB2013,DBRR2002,  EO1992, GGZ2007,  GMWP2012, HYDY2016,IM2016, KAFHKKR2015, KOCZR2018, LRMR2016, PPRZMHBS2015, PPRZMHBS2016, RSM2014, RSZ2018,  TT2017, WLGNMR2017}. In \textcolor{black}{such} problems, the varying physical properties between the different regions result in potentially large modeling differences. In many instances, a material which comprises a relatively small portion of the domain may feature significantly more complex physics and\textcolor{black}{, therefore, it may} require a higher mesh resolution than in the remainder of the domain. Common examples of such problems with high industrial interest are thermal problems in additive manufacturing  (where the nonlinear phase transformation phenomena occur only along a thin strip on the top of the domain)\cite{ BRDP2018, GMWP2012, HYDY2016,IM2016, KAFHKKR2015, KOCZR2018, LRMR2016, PPRZMHBS2015, PPRZMHBS2016, RSM2014, RSZ2018, TT2017,WLGNMR2017}, or fluid flow problems through immersed membranes (where the flow properties change inside a subregion of the domain) \cite{LWPLK2009, LL2003, LL1997, PPRP2005}. A simple diagram of such a physical situation is shown in Fig. \ref{fig:ImmCircIntro}.

\par While using fine mesh throughout the \textcolor{black}{whole} domain may give an acceptable numerical solution, the computational cost of such an approach may be prohibitive. In contrast, a mesh that is coarse in one region and fine in another, while computationally more attractive, may be difficult to generate and is also known to cause difficulties with preconditioning \cite{DBRR2002, DWZ2009, Elman, KHX2014}. Additionally, in some instances the boundary of the subregion may be geometrically complicated, or may change in time, requiring frequent \textcolor{black}{remeshing or the use of complicated local refinement and derefinement techniques} \cite{IM2016, KOCZR2018, PPRZMHBS2015, PPRZMHBS2016, RSM2014, WLGNMR2017}.
\par \textcolor{black}{We decided to approach such problems adopting a different methodology inspired by the Fat boundary method (see \cite{BMM2005, BIM2011, Maury2001}), an algorithm originally designed to handle geometries with holes in which generating conformal meshes is difficult.} The Fat boundary method works by solving problems on different meshes and coupling the different solutions by the \textcolor{black}{introduction of point forces within the domain}. In the present work, we develop a related technique that allows us to solve problems with localized nonhomogeneous features by using two separate meshes. In regions that require a higher mesh resolution, we solve the problem on a fine separate mesh; then, as in the \textcolor{black}{Fat boundary method}, we reinsert this solution \textcolor{black}{on} the coarse mesh as a forcing term. This approach can be viewed as \textcolor{black}{a generalization} of the \textcolor{black}{Fat boundary method}, since the subregions are not holes in the geometry but areas in which the material properties differ.  

\begin{figure}
\centering
\includegraphics[scale=.76]{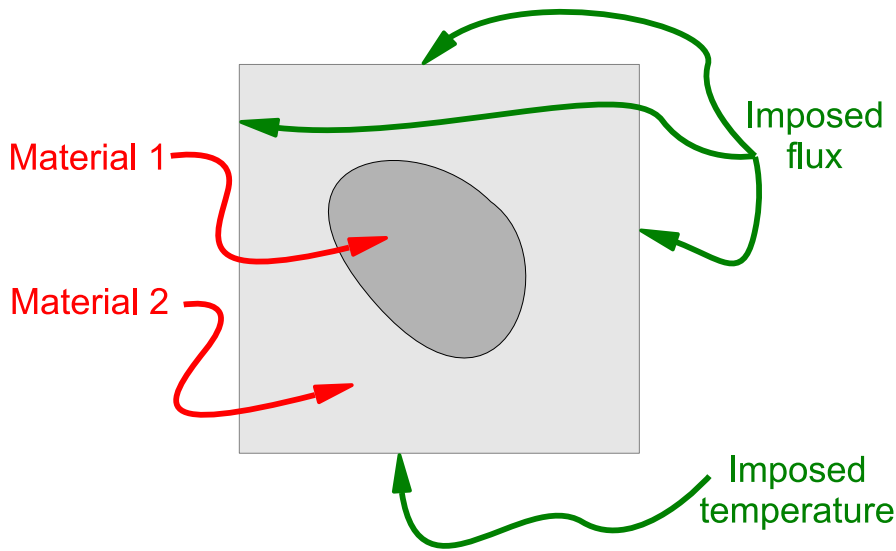} \caption{Example physical situation. The domain features and imposed flux along one region and imposed temperature on the others. The material properties differ within the domain.}\label{fig:ImmCircIntro}
\end{figure}

\par \textcolor{black}{We structure the paper as follows. We first introduce the model problems and derive split-problem formulations, in which we express the model problems (originally defined on a single mesh) as a pair of problems on two different meshes with coupling conditions.} Next, \textcolor{black}{we use the split formulation to naturally derive} two-level methods at both the continuous and discrete levels. We then present numerical examples of \textcolor{black}{the proposed approach for both steady and unsteady problems} and discuss the results. We conclude by discussing points for further improvement and future research directions. 


\section{Problem Splitting}

\begin{figure}
\centering
\includegraphics[scale=.76]{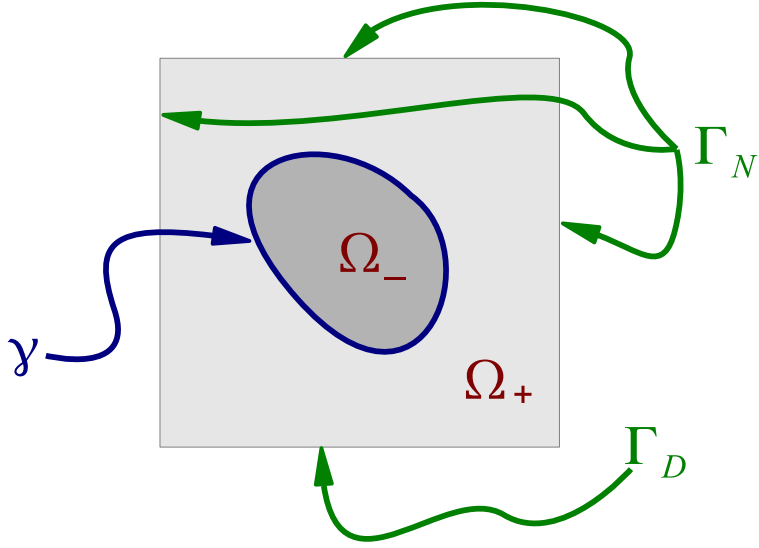} \caption{Example geometry for model problem. The material properties differ in $\Omega_-$ and $\Omega_+$.}\label{fig:ImmCirc1}
\end{figure}
In the following we focus on steady and unsteady thermal problems, which are often characterized by the presence of local features, such as a nonhomogeneous material inclusion or local nonlinear phase transition phenomena. Accordingly, in this section we formally deriv\textcolor{black}{e}  split formulation\textcolor{black}{s} \textcolor{black}{for the study of} such problems. The \textcolor{black}{proposed formulations will then naturally} form the basis for the two-level algorithm presented later in the work.
\par  We consider a domain $\Omega$ with a subregion $\Omega_- \subset \Omega$, as shown in Fig \ref{fig:ImmCirc1}. Letting $\Omega_+ = \Omega \setminus \overline \Omega_-$, $\gamma = \partial \Omega_-$, and denoting the unknown temperature field as $u$, the steady thermal problem is given by:
\begin{align}\begin{split}\label{Model1}
	-\nabla \cdot \left(\beta \nabla u \right) &= f \quad \text{in }  \Omega\SB{\setminus\gamma} \\
	u &= T_0 \quad \text{on } \Gamma_{D} \\
	\beta \frac{\partial u}{\partial \nn} &= q \quad \text{on } \Gamma_{N}\\
\SB{\jump{\beta \nabla u \cdot \nn}}&\SB{= 0 \quad \text{on }\gamma}\\
\SB{\jump{u}}&\SB{= 0 \quad \text{on }\gamma}
\end{split}\end{align}
\textcolor{black}{w}ith the unsteady variant defined as:
\begin{align}\begin{split}\label{Model2}
	\rho\frac{\partial u}{\partial t} -\nabla \cdot \left(\beta\nabla u \right) &= f \quad \text{in }  \Omega \times \left\lbrack t_0,\,t_{end} \right\rbrack \\
	u &= T_0 \quad \text{on } \Gamma_{D} \times \left\lbrack t_0,\,t_{end} \right\rbrack \\
	\beta \frac{\partial u}{\partial \nn} &= q(t) \quad \text{on } \Gamma_{N} \times \left\lbrack t_0,\,t_{end} \right\rbrack \\
	u &= u_0 \quad \text{at } t=t_0\textcolor{black}{.}\\
	\SB{\jump{\beta \nabla u \cdot \nn}}&\SB{= 0 \quad \text{on }\gamma}\ \\
	\SB{\jump{u}}&\SB{= 0 \quad \text{on }\gamma} 
\end{split}\end{align}
\textcolor{black}{where }$\partial u\big/\partial \nn$ denotes the normal derivative, $q$ is a heat flux, \SB{ $\jump{\cdot}$ denotes the jump across $\gamma$}, \textcolor{black}{$\beta$ denotes the heat conductivity, defined as:
\begin{eqnarray}\label{betaDef}
\beta := \begin{cases}
\beta_{+} \quad \text{in }  \Omega_+ \\
\beta_{-} \quad \text{in } \Omega_{-}
\end{cases},
\end{eqnarray}
and $\rho$  denotes the heat capacity,  defined as:
\begin{eqnarray}
\quad \rho := \begin{cases}
\rho_{+} \quad \text{in } \Omega_+ \\
\rho_{-} \quad \text{in } \Omega_{-}
\end{cases} \end{eqnarray}}
 \par In general, $\beta_{+}$, $\beta_{-}$, $\rho_+$ and $\rho_-$ are functions of time and space. However, as the constant-coefficient case greatly simplifies the mathematical formulation and analysis of the method, we will start by considering this simpler but important case, providing a more general analysis in an appendix at end of the present work.
 \par As the methods discussed here bear a strong relationship to the previously discussed Fat boundary method, we follow the same notation used in \cite{ BMM2005, BIM2011, Maury2001} and denote the trace of $H^1\left(\Omega\right)$ on $\gamma$ as $H^{1/2}\left(\gamma \right)$, with the corresponding dual spaces given by $H^{-1}\left(\Omega\right)$ and  $H^{-1/2}\left( \gamma \right)$ respectively. 
 \par For $\eta \in H^{-1/2}\left(\gamma \right)$, we define $\eta \delta_{\gamma} \in H^{-1}\left(\Omega\right) = H^1(\Omega)'$ as the linear functional such that:
  \begin{align}\label{deltaDef}
  	\int_{\Omega} \left(\eta \delta_{\gamma}\right)w = \int_{\gamma} \eta w \quad\quad \forall w \in H^1\left(\Omega\right),
  \end{align}
\textcolor{black}{and, in the spirit of the Fat boundary method, we propose to split Problem (\ref{Model1}) with constant $\beta_{+}$ and $\beta_{-}$ into two subproblems as follows:}
 \begin{align}\begin{split}\label{pb1NoAlg}
  	-\nabla \cdot \left(\beta_{+} \nabla {u} \right)  &= f\big|_{\Omega_+} +\frac{\beta_+}{\beta_-}f\big|_{\Omega_-}    + \left\lbrack \beta_{+} \frac{\partial \widetilde{u} }{\partial \nn}  \delta_{\gamma} - \beta_{-} \frac{\partial \widetilde{u}}{\partial \nn} \delta_{\gamma} \right\rbrack  \text{  in  } \Omega \\
	{u} &= T_0 \text{  on  } \Gamma_D \\
	\beta_{+} \frac{\partial {u}}{\partial \nn} &= q \text{  on  } \Gamma_N
\end{split} \end{align}
 and
\begin{align}\begin{split}\label{pb2NoAlg}
	-\nabla \cdot \left(\beta_{-} \nabla \widetilde{u} \right) &= f \text{  in  } \Omega_{-} \\
	\widetilde{u} &= u \text{  on  } \gamma .\end{split}\end{align}
\textcolor{black}{We now prove a theorem, stating the equivalence between the model Problem (\ref{Model1}) and Problems (\ref{pb1NoAlg}) and (\ref{pb2NoAlg}):}    
\begin{theorem}\label{thm:cons}
A temperature distribution $u$ solves \textcolor{black}{Problem (\ref{Model1})} with constant $\beta_{+}$ and $\beta_-$ if and only if the corresponding solution pair $(u, \widetilde{u})$ solves \textcolor{black}{Problems (\ref{pb1NoAlg}) and (\ref{pb2NoAlg})}.
\end{theorem}
\begin{proof} Let $u$ be a solution of (\ref{Model1}). Define $\widetilde{u}$ as:
\begin{equation}\label{uTildeDef}
\widetilde u := u|_{\Omega_-}.
\end{equation}
This of course implies:
\begin{align}\label{uTildeExp}
	-\nabla \cdot \left(\beta_-\nabla\widetilde{u}\right) &= f \text{  in  } \Omega_-.
\end{align}
Then for $u \in H^{1}\left(\Omega\right)$ and $\varphi \in C_{0}^{\infty}\left(\Omega\right)$:
\begin{align}\begin{split}\label{proofStep2}
\int_{\Omega} \beta \nabla{u} \cdot \nabla{\varphi} &= \int_{\Omega} \beta_+ \nabla{u} \cdot \nabla{\varphi} - \int_{\Omega_-} \beta_+ \nabla u \cdot \nabla \varphi + \int_{\Omega_-} \beta_- \nabla u \cdot \nabla \varphi  \\
	&= \int_{\Omega} \beta_+ \nabla{u} \cdot \nabla{\varphi} - \int_{\Omega_-} \beta_+ \nabla \widetilde{u} \cdot \nabla \varphi + \int_{\Omega_-} \beta_- \nabla \widetilde{u} \cdot \nabla \varphi  
\end{split}\end{align}
Since $u$ satisfies (\ref{Model1}) weakly in $\Omega$ it follows from \textcolor{black}{(\ref{proofStep2}) that:
\begin{align}\label{moreFormal}
	\int_{\Omega} \beta_+ \nabla{u} \cdot \nabla{\varphi} - \int_{\Omega_-} \beta_+ \nabla \widetilde{u} \cdot \nabla \varphi + \int_{\Omega_-} \beta_- \nabla \widetilde{u} \cdot \nabla \varphi &= \int_{\Omega} f\varphi \text{  in  } \Omega_-.
\end{align}}
\par In particular, $-\beta_- \Delta \widetilde{u} = f$ weakly in $\Omega_-$ (since $\beta_-$ is constant). Integrating the integrals in $\Omega_{-}$ on the right hand side of (\ref{proofStep2}) by parts yields:
\begin{align}\begin{split}\label{proofStep3}\MoveEqLeft[1] 
- \int_{\Omega_-} \beta_+ \nabla \widetilde{u} \cdot \nabla \varphi + \int_{\Omega_-} \beta_- \nabla \widetilde{u} \cdot \nabla \varphi \\ \quad{}& = \int_{\Omega_-} \beta_+ \Delta \widetilde{u} \varphi - \int_{\gamma} \beta_{+} \frac{\partial \widetilde{u}}{\partial \nn} \varphi - \int_{\Omega_-} \beta_- \Delta \widetilde{u} \varphi + \int_{\gamma} \beta_{-} \frac{\partial \widetilde{u}}{\partial \nn} \varphi \\
&= - \int_{\Omega_-} \frac{\beta_+}{\beta_-} f \varphi - \int_{\gamma} \beta_{+} \frac{\partial \widetilde{u}}{\partial \nn} \varphi + \int_{\Omega_-} f \varphi + \int_{\gamma} \beta_{-} \frac{\partial \widetilde{u}}{\partial \nn} \varphi \end{split}\end{align}	
Combining (\ref{proofStep3}) with (\ref{proofStep2}) then gives: 
\begin{equation}
    \int_{\Omega} \beta_+ \nabla{u} \cdot \nabla{\varphi} = \int_{\Omega_+} f\varphi +  \frac{\beta_+}{\beta_-} \int_{\Omega_-} f\varphi + \beta_+\int_{\gamma} \frac{\partial \widetilde{u} }{\partial \nn} \varphi -  \beta_-\int_{\gamma} \frac{\partial \widetilde{u} }{\partial \nn} \varphi,
\end{equation}
\textcolor{black}{and, after backward integration over the integral in $\Omega$,}
\begin{equation}
-\nabla \cdot (\beta_+ \nabla u) = f\big|_{\Omega_+} + \frac{\beta_+}{\beta_-} f\big|_{\Omega_-} + (\beta_+-\beta_-) \frac{\partial \widetilde{u}}{\partial \nn}\delta_{\gamma} \text{ in } \Omega.
\end{equation}
which was to be shown. 

\

Let us now assume that  $(u,\widetilde{u})$ is a solution to the coupled problem. We start by proving that $\widetilde u = u $ in $\Omega_-$. In fact, it is not difficult to verify that $u - \widetilde u$ solves the equation
\[
-\nabla \cdot \beta_- \nabla (u - \widetilde u) = 0, \quad \text{in }\Omega_-, \qquad u - \widetilde u = 0, \quad \text{on }\gamma,
\]
that has the unique solution $u - \widetilde u = 0$. As a consequence, one has that
\[
-\nabla \cdot \beta u = f, \quad \text{in } \Omega \setminus \gamma.
\]

On the other hand, $u$ verifies
\[
\jump{\beta^+ \nabla u \cdot \nn} = (\beta^+ - \beta^ - ) \frac{\partial \widetilde u}{\partial \nn} = (\beta^+ - \beta^ - ) \frac{\partial u}{\partial \nn}.
\]
This implies $\jump{\beta \nabla u \cdot \nn} = 0$.
\end{proof}

\textcolor{black}{For the unsteady thermal problem, we proceed with the same approach and split (\ref{Model2}) into two subproblems as:} 
 \begin{align}\begin{split}\label{pb1NoAlgUnst}\MoveEqLeft[4]
  \rho_+\frac{\partial u}{\partial t} + \left(\frac{\rho_- \beta_+}{\beta_-} -  \rho_+\right) \frac{\partial \widetilde{u} }{\partial t}\big|_{\Omega_-}  	-\nabla \cdot \left(\beta_{+} \nabla {u} \right)  \\ \quad{}& = f\big|_{\Omega_+} +\frac{\beta_+}{\beta_-}f\big|_{\Omega_-}    + \left\lbrack \beta_{+} \frac{\partial \widetilde{u} }{\partial \nn}  \delta_{\gamma} - \beta_{-} \frac{\partial \widetilde{u}}{\partial \nn} \delta_{\gamma} \right\rbrack  \text{  in  } \Omega \times \left\lbrack t_0,\,t_{end} \right\rbrack \\
	{u} &= T_0 \text{  on  } \Gamma_{D} \times \left\lbrack t_0,\,t_{end} \right\rbrack \\
\beta_+ \frac{\partial u}{\partial \nn} &= q(t) \quad \text{on } \Gamma_{N} \times \left\lbrack t_0,\,t_{end} \right\rbrack \\
\end{split} \end{align}
 and
\begin{align}\begin{split}\label{pb2NoAlgUnst}
	\rho_-\frac{\partial \widetilde{u}}{\partial t} -\nabla \cdot \left(\beta_{-} \nabla \widetilde{u} \right) &= f \text{  in  } \Omega_{-}(t) \times \left\lbrack t_0,\,t_{end}\right\rbrack \\
	\widetilde{u} &= u \text{  on  } \gamma(t) \times \left\lbrack t_0,\,t_{end}\right\rbrack, \end{split}\end{align}
	which clearly need to be properly complemented by suitable initial conditions.
\begin{theorem}\label{thm:cons2}
A temperature distribution $u(t)$ solves \textcolor{black}{Problem (\ref{Model2})} with constant $\beta_{+}$, $\beta_-$, $\rho_{+}$ and $\rho_-$ if and only if the corresponding solution pair $(u(t), \widetilde{u}(t) )$ solves \textcolor{black}{Problems (\ref{pb1NoAlgUnst}) and (\ref{pb2NoAlgUnst}).}
\end{theorem}   
\begin{proof}
This proof strongly resembles \textcolor{black}{the one adopted in Theorem \ref{thm:cons}}. As previously, let $u(t)$ be a solution of (\ref{Model2}). The dependence \textcolor{black}{of $u$ and $\tu$ on $t$ is assumed for the remainder of this proof, even if not explicitly noted}. We again define:
\begin{eqnarray}\label{uTildeDef2}
\widetilde{u} := u|_{\Omega_-}
\end{eqnarray}
Then for $u \in H^{1}\left(\Omega\right)$ and $\varphi \in C_{0}^{\infty}\left(\Omega\right)$:
\begin{align}\begin{split}\label{proofStep2Unst}\MoveEqLeft[4]
\int_{\Omega} \rho \frac{\partial u}{\partial t}\varphi + \int_{\Omega} \beta \nabla{u} \cdot \nabla{\varphi} \\ \quad{}& = \int_{\Omega} \rho_+ \frac{\partial u}{\partial t} \varphi - \int_{\Omega_-} \rho_+ \frac{\partial u}{\partial t} \varphi + \int_{\Omega_-} \rho_- \frac{\partial u}{\partial t} \varphi \\
\quad{}& +  \int_{\Omega} \beta_+ \nabla{u} \cdot \nabla{\varphi} - \int_{\Omega_-} \beta_+ \nabla u \cdot \nabla \varphi + \int_{\Omega_-} \beta_- \nabla u \cdot \nabla \varphi \\
\quad{}& = \int_{\Omega} \rho_+ \frac{\partial u}{\partial t} \varphi - \int_{\Omega_-} \rho_+ \frac{\partial \widetilde{u}}{\partial t} \varphi + \int_{\Omega_-} \rho_- \frac{\partial \widetilde{u}}{\partial t} \varphi \\
\quad{}& +  \int_{\Omega} \beta_+ \nabla{u} \cdot \nabla{\varphi} - \int_{\Omega_-} \beta_+ \nabla \widetilde{u} \cdot \nabla \varphi + \int_{\Omega_-} \beta_- \nabla \widetilde{u} \cdot \nabla \varphi \\
\end{split}\end{align}
	\textcolor{black}{Since $u$ satisfies (\ref{Model2}) weakly in $\Omega$, it follows from (\ref{proofStep2Unst}) that 
\begin{align}\begin{split}\label{moreClarity2}\MoveEqLeft[4]	
	\int_{\Omega} \rho_+ \frac{\partial u}{\partial t} \varphi - \int_{\Omega_-} \rho_+ \frac{\partial \widetilde{u}}{\partial t} \varphi + \int_{\Omega_-} \rho_- \frac{\partial \widetilde{u}}{\partial t} \varphi \\
\quad{}& +  \int_{\Omega} \beta_+ \nabla{u} \cdot \nabla{\varphi} - \int_{\Omega_-} \beta_+ \nabla \widetilde{u} \cdot \nabla \varphi + \int_{\Omega_-} \beta_- \nabla \widetilde{u} \cdot \nabla \varphi \\
\quad{}&= \int_{\Omega} f \varphi  \text{  in  } \Omega_-,
\end{split}\end{align}
while definition (\ref{uTildeDef2}) ensures:
\begin{align}\label{lapRelation}
-\beta_- \Delta \widetilde{u} = f - \rho_- \frac{\partial \widetilde{u} }{\partial t}\text{  in  } \Omega_-.
\end{align}}
Here we have used the fact that $\beta_+$ and $\beta_-$ are constant. Integration by parts and (\ref{lapRelation}) then give:
	  \begin{align}\begin{split}\label{proofStep3Unst}\MoveEqLeft[1] 
- \int_{\Omega_-} \beta_+ \nabla \widetilde{u} \cdot \nabla \varphi + \int_{\Omega_-} \beta_- \nabla \widetilde{u} \cdot \nabla \varphi \\ \quad{}& = \int_{\Omega_-} \beta_+ \Delta \widetilde{u} \varphi - \int_{\gamma} \beta_{+} \frac{\partial \widetilde{u}}{\partial \nn} \varphi - \int_{\Omega_-} \beta_- \Delta \widetilde{u} \varphi + \int_{\gamma} \beta_{-} \frac{\partial \widetilde{u}}{\partial \nn} \varphi \\
&= - \int_{\Omega_-} \frac{\beta_+}{\beta_-} f \varphi +\int_{\Omega_-} \frac{\rho_- \beta_+}{\beta_-} \frac{\partial \widetilde{u}}{\partial t} \varphi - \int_{\gamma} \beta_{+} \frac{\partial \widetilde{u}}{\partial \nn} \varphi + \int_{\Omega_-} f \varphi - \int_{\Omega_-} \rho_- \frac{\partial \widetilde{u}}{\partial t}  + \int_{\gamma} \beta_{-} \frac{\partial \widetilde{u}}{\partial \nn} \varphi \end{split}\end{align}	
Combining (\ref{proofStep3Unst}) and (\ref{proofStep2Unst}) and simplifying gives:
\begin{align}\begin{split}\label{proofStep4Unst}\MoveEqLeft[1]
\int_{\Omega} \rho_+ \frac{\partial u}{\partial t} \varphi + \int_{\Omega} \beta_+ \nabla u \cdot \nabla \varphi \\ \quad{}& = 
\left(\rho_+ - \frac{\rho_- \beta_+}{\beta_-} \right)\int_{\Omega_-} \frac{\partial \widetilde{u} }{\partial t} \varphi + \int_{\Omega_+} f\varphi + \frac{\beta_+}{\beta_-} \int_{\Omega_-} f\varphi + \beta_+\int_{\gamma} \frac{\partial \widetilde{u} }{\partial \nn} \varphi -  \beta_-\int_{\gamma} \frac{\partial \widetilde{u} }{\partial \nn} \varphi
\end{split}\end{align}
\textcolor{black}{which, after backward integration over $\Omega$, implies:}
\begin{align}\begin{split}\label{resultUnst}\MoveEqLeft[1]
\rho_+ \frac{\partial u}{\partial t} + (\frac{\rho_- \beta_+}{\beta_-} -  \rho_+) \frac{\partial \widetilde{u} }{\partial t}\big|_{\Omega_-}  -\nabla \cdot (\beta_+ \nabla u) \\ \quad{}& = f\big|_{\Omega_+} + \frac{\beta_+}{\beta_-} f\big|_{\Omega_-} + (\beta_+-\beta_-) \frac{\partial \widetilde{u}}{\partial \nn}\delta_{\gamma} \text{ in } \Omega \times \left\lbrack t_0,\,t_{end} \right\rbrack \end{split}\end{align}
\textcolor{black}{w}hich was to be shown. 
\par Now suppose the pair $(u,\tu)$ solves Problem (\ref{pb1NoAlgUnst}) and Problem (\ref{pb2NoAlgUnst}). We seek to demonstrate that $u-\tu$ solves the equation:
\begin{align}\begin{split}\label{verifPbmK}
	\rho_+\frac{\partial (u-\tu)}{\partial t} -\nabla \cdot \left(\beta_{+} \nabla (u-\tu) \right) &= 0 \text{  in  } \Omega_{-}(t) \times \left\lbrack t_0,\,t_{end}\right\rbrack \\
	u-\widetilde{u} &= 0 \text{  on  } \gamma(t) \times \left\lbrack t_0,\,t_{end}\right\rbrack 
\end{split}\end{align}
with initial condition $(u-\tu)(t_0)=0$. One may verify that (\ref{verifPbmK}) has a unique solution 0 in $\Omega_-(t)$, and hence if $u-\tu$ satisfies (\ref{verifPbmK}), $u=\tu$ in $\Omega_-(t)$. 
\par For all $\varphi \in H_0^1(\Omega_-)$ we have\footnote{Note that as (\ref{verifPbmK}) is zero on the boundary, we may demand this on our test space as well without loss of generality.}:
\begin{align}\begin{split}\label{unstCons1}
	\rho_+\int_{\Omega_-}\frac{\partial \tu}{\partial t}\varphi - \beta_+ \int_{\Omega_-} \Delta \tu \, \varphi &= \rho_+\int_{\Omega_-} \frac{\partial u}{\partial t}\varphi - \beta_+\int_{\Omega_-} \Delta u \varphi.
\end{split}\end{align}
From (\ref{lapRelation}), 
\begin{align}\begin{split}\label{unstCons2}
	-\beta_+\int_{\Omega_-} \Delta \tu \, \varphi &= \frac{\beta_+}{\beta_-} \int_{\Omega_-} f \varphi - \frac{\rho_-\beta_+}{\beta_-} \int_{\Omega_-} \frac{\partial \tu}{\partial t}\varphi,
\end{split}\end{align}
 which substituting into (\ref{unstCons1}) and applying the definition of Problem (\ref{pb1NoAlgUnst}) gives:
 \begin{align}\begin{split}\label{unstCons3}
\left(\rho_+ -\frac{\rho_-\beta_+}{\beta_-}\right) \int_{\Omega_-}\frac{\partial \tu}{\partial t}\varphi + \frac{\beta_+}{\beta_-} \int_{\Omega_-} f \varphi &= \rho_+\int_{\Omega_-} \frac{\partial u}{\partial t}\varphi - \beta_+\int_{\Omega_-} \Delta u \varphi \\
\rho_+\int_{\Omega_-} \frac{\partial u}{\partial t}\varphi - \beta_+\int_{\Omega_-} \Delta u \varphi &= \rho_+\int_{\Omega_-} \frac{\partial u}{\partial t}\varphi - \beta_+\int_{\Omega_-} \Delta u \varphi
\end{split}\end{align}
establishing $\tu-u=0$ in $\Omega_-$. As in the proof of Theorem 2.1,
\begin{align} 
\jump{\beta^+ \nabla u \cdot \nn} = (\beta^+ - \beta^ - ) \frac{\partial \widetilde u}{\partial \nn} = (\beta^+ - \beta^ - ) \frac{\partial u}{\partial \nn}, \end{align}
implying $\jump{\beta \nabla u \cdot \nn}=0$ and hence the result.
 \end{proof}
When $\beta$ is non-constant, the splitting of the steady problem (\ref{Model1}) is given by:
 \begin{align}\begin{split}\label{pb1NoAlgNonCons}\MoveEqLeft[1]
  	-\nabla \cdot \left(\beta_{+} \nabla {u} \right) \\ \quad{}& =  f\big|_{\Omega_+} +\frac{\beta_+}{\beta_-}f\big|_{\Omega_-}
  	+\frac{\beta_+}{\beta_-}\nabla \widetilde{u} \cdot \nabla \beta_- \big|_{\Omega_-}
  	- \nabla \widetilde{u} \cdot \nabla \beta_+ \big|_{\Omega_-}
  	    +  \left(\beta_{+}-\beta_-\right) \frac{\partial \widetilde{u} }{\partial \nn}  \delta_{\gamma}   \text{  in  } \Omega \\
	{u} &= T_0 \text{  on  } \Gamma_D \\
	\beta_{+} \frac{\partial {u}}{\partial \nn} &= q \text{  on  } \Gamma_N
\end{split} \end{align}
 and

\begin{align}\begin{split}\label{pb2NoAlgNonCons}
	-\nabla \cdot \left(\beta_{-} \nabla \widetilde{u} \right) &= f \text{  in  } \Omega_{-} \\
	\widetilde{u} &= u \text{  on  } \gamma .\end{split}\end{align}
Similarly, for Problem (\ref{Model2}) with non-constant $\beta$ and $\rho$, the splitting reads: 
 \begin{align}\begin{split}\label{pb1NoAlgNonConsUnst}\MoveEqLeft[1]
  \rho_+\frac{\partial u}{\partial t} + \left(\frac{\rho_-\beta_+}{\beta_-}-\rho_+ \right)	\frac{\partial \tu}{\partial t}\big|_{\Omega_-} \\ \quad{}& - \nabla \cdot \left(\beta_{+} \nabla {u} \right)  =  f\big|_{\Omega_+} +\frac{\beta_+}{\beta_-}f\big|_{\Omega_-}
  	+\frac{\beta_+}{\beta_-}\nabla \widetilde{u} \cdot \nabla \beta_- \big|_{\Omega_-}
  	\\ \quad{}&  - \nabla \widetilde{u} \cdot \nabla \beta_+ \big|_{\Omega_-}
  	    +  \left(\beta_{+}-\beta_-\right) \frac{\partial \widetilde{u} }{\partial \nn}  \delta_{\gamma}   \text{  in  } \Omega \times \lbrack t_0,\,t_{end}\rbrack \\
	{u} &= T_0 \text{  on  } \Gamma_D \times \lbrack t_0,\,t_{end}\rbrack  \\
	\beta_{+} \frac{\partial {u}}{\partial \nn} &= q(t) \text{  on  } \Gamma_N \times \lbrack t_0,\,t_{end}\rbrack \\
\end{split} \end{align}
 and

\begin{align}\begin{split}\label{pb2NoAlgNonConsUnst}
\rho_-\frac{\partial \tu}{\partial t}	-\nabla \cdot \left(\beta_{-} \nabla \widetilde{u} \right) &= f \text{  in  } \Omega_{-} \\
	\widetilde{u} &= u \text{  on  } \gamma, \end{split}\end{align}
	which must be also equipped with suitable initial conditions.
\begin{theorem}\label{thm:nonconsSteady}
A temperature distribution $u$ solves Problem (\ref{Model1}) for $\beta_- \neq 0 \in \Omega_-$, \newline $\frac{\beta_+}{\beta_-} \in H^1(\Omega_-)$ if and only if the corresponding solution pair $(u, \tu)$ solves Problems (\ref{pb1NoAlgNonCons}) and (\ref{pb2NoAlgNonCons}). 	
\end{theorem}
\begin{theorem}\label{thm:nonconsUnsteady}
A temperature distribution solves Problem (\ref{Model2}) for $\beta_- \neq 0 \in \Omega_-$, \newline $\frac{\beta_+}{\beta_-} \in H^1(\Omega_-)$, $\frac{\rho_-\beta_+}{\beta_-} \in H^1(\Omega_-)$ if and only if the corresponding solution pair $(u, \tu)$ solves Problems (\ref{pb1NoAlgNonConsUnst}) and (\ref{pb2NoAlgNonConsUnst}). 	
\end{theorem}
Proofs of Theorems \ref{thm:nonconsSteady} and \ref{thm:nonconsUnsteady} are similar to those shown previously but somewhat more involved. They are provided in an appendix at the end of this work.
\section{Two-Level Method}

 We will now use the coupled formulations (\ref{pb1NoAlg})-(\ref{pb2NoAlg}) and (\ref{pb1NoAlgUnst})-(\ref{pb2NoAlgUnst}) introduced in the previous section to derive two-level methods for solving the steady (\ref{Model1}) and unsteady (\ref{Model2}) thermal problems.
 \par Focusing initially on steady problems, we propose to solve (\ref{pb1NoAlg}) and (\ref{pb2NoAlg}) adopting an iterative scheme using an approach similar to the \textcolor{black}{one proposed in \cite{BMM2005, BIM2011, Maury2001}} and detailed in the following algorithm:
\begin{alg}\label{TwoLevelHeat}
The two-level algorithm for Problem (\ref{Model1}) is given by: \\ \ \\
Step 1: Obtain initial temperature distribution $u^0$ by solving on the entire domain $\Omega$: 
\begin{align}\begin{split}\label{pb1}
	-\nabla \cdot \left(\beta_{+} \nabla u^0 \right) &= f \text{ in } \Omega \\
	u^0 &= T_0 \text{  on  } \Gamma_D \\
	\beta_{+} \frac{\partial u^0}{\partial \nn} &= q \text{  on  } \Gamma_N
	\end{split}\end{align}
Step k\textcolor{black}{: perform the following 4 steps:} \\

k.1 Obtain intermediate temperature distribution $\widetilde{u}^k$ by solving on the subdomain $\Omega_-$:
\begin{align}\begin{split}\label{pb2}
	-\nabla \cdot \left(\beta_{-} \nabla \widetilde{u}^k \right) &= f \text{  in  } \Omega_{-} \\
	\widetilde{u}^k &= u^{k-1} \text{  on  } \gamma \\
	\end{split}\end{align} 

k.2 Obtain the temperature distribution $\widehat{u}^k$ by solving on the entire domain $\Omega$: 
\begin{align}\begin{split}\label{pb3}
	-\nabla \cdot \left(\beta_{+} \nabla \widehat{u}^k \right)  &= f\big|_{\Omega_+} + \frac{\beta_+}{\beta_-}f\big|_{\Omega_-} + \left\lbrack \beta_{+} \frac{\partial \widetilde{u}^k }{\partial \nn} \delta_{\gamma} - \beta_{-} \frac{\partial \widetilde{u}^k}{\partial \nn} \delta_{\gamma} \right\rbrack  \text{  in  } \Omega \\
	\widehat{u}^k &= T_0 \text{  on  } \Gamma_D \\
	\beta_{+} \frac{\partial \widehat{u}^k}{\partial \nn} &= q \text{  on  } \Gamma_N
	\end{split}\end{align}

k.3 \textcolor{black}{Perform a }relaxation step to obtain a temperature distribution $u^k$:
\begin{align}\label{relax}
	u^k = \theta \widehat{u}^k + \left(1-\theta \right)u^{k-1}, \quad \theta \in (0, \,1\rbrack.
\end{align}

k.4 Check convergence. If met, terminate iteration. Otherwise repeat step k.
\end{alg}
We note that the influence of $\beta_-$ is introduced by the forcing term in step k.2. 
\par Under-relaxation is used here, as this algorithm \textcolor{black}{is iterative in nature and} can suffer from instability, particularly when $\beta_- >> \beta_+$. The convergence analyses for the related methods shown in \cite{BMM2005, BIM2011, Maury2001} suggest that the relaxation step k.3 is important for stability and convergence, with a way to optimally select $\theta$ under specific assumptions demonstrated in \cite{Maury2001}. A similar analysis for Algorithm \ref{TwoLevelHeat} is an important subject for future work.
\par For the unsteady problem (\ref{Model2}), we \textcolor{black}{introduce a } time-discrete algorithm\textcolor{black}{, based on a backward Euler method}. \textcolor{black}{Since at each time step we must use an iterative scheme, to avoid confusion, we use the index $n$ to denote the time steps and $k$ to indicate the iterates within a time step. With this notation, the algorithm reads:}  \begin{alg}\label{TwoLevelHeatUnst}
The two-level algorithm for Problem (\ref{Model2}) at a time step $t_n$ is given by: \\ \ \\
Step 1: Obtain initial temperature distribution $u_n^0$ at time $t_n$ by solving on the entire domain $\Omega$: 
\begin{align}\begin{split}\label{pb1}
 \frac{\rho_+}{\Delta t}\left(u_n^0 - u_{n-1} \right) 	-\nabla \cdot \left(\beta_{+} \nabla u_n^0 \right) &= f_n \text{  in  } \Omega \\
	u_n^0 &= T_0 \text{  on  } \Gamma_D \\
	\beta_{+} \frac{\partial u_n^0}{\partial \nn} &= q_n \text{  on  } \Gamma_N
	\end{split}\end{align}
Step k consists of the following 4 steps: \\

k.1 Obtain intermediate temperature distribution $\widetilde{u}^k$ by solving on the subdomain $\Omega_-$:
\begin{align}\begin{split}\label{pb2}
\frac{\rho_-}{\Delta t}\left(\widetilde{u}_n^k - u_{n-1} \right)	-\nabla \cdot \left(\beta_{-} \nabla \widetilde{u}_n^k \right) &= f_n \text{  in  } \Omega_{-} \\
	\widetilde{u}_n^k &= u_n^{k-1} \text{  on  } \gamma \\
	\end{split}\end{align} 

k.2 Obtain the temperature distribution $\widehat{u}_n^k$ by solving on the entire domain $\Omega$: 
\begin{align}\begin{split}\label{pb3}
\MoveEqLeft[1] 
\frac{\rho_+}{\Delta t}\left(\widehat{u}_n^k - u_{n-1} \right)	-\nabla \cdot \left(\beta_{+} \nabla \widehat{u}_n^k \right) \\ \quad{}&  = f_n\big|_{\Omega_+} + \frac{\beta_+}{\beta_-}f_n\big|_{\Omega_-} + \left\lbrack \beta_{+} \frac{\partial \widetilde{u}_n^k }{\partial \nn} \delta_{\gamma} - \beta_{-} \frac{\partial \widetilde{u}_n^k}{\partial \nn} \delta_{\gamma} \right\rbrack \\ \quad{}& + \left(\rho_+ -\frac{\rho_- \beta_+}{\beta_-}\right) \left\lbrack\frac{1}{\Delta t}\left(\widetilde{u}_n^k - u_{n-1} \right)\bigg|_{\Omega_-}\right\rbrack  \text{  in  } \Omega \\
	\widehat{u}_n^k &= T_0 \text{  on  } \Gamma_D \\
	\beta_{+} \frac{\partial \widehat{u}_n^k}{\partial \nn} &= q_n \text{  on  } \Gamma_N
	\end{split}\end{align}

k.3 \textcolor{black}{Perform a relaxation step to obtain the temperature distribution} $u_n^k$:
\begin{align}\label{relax}
	u_n^k = \theta \widehat{u}_n^k + \left(1-\theta\right)u_n^{k-1}, \quad \theta \in (0, \,1\rbrack.
\end{align}

k.4 Check convergence. If met, terminate iteration. Otherwise repeat step k.
\end{alg}

\section{Finite Element formulations}
We now present the algorithms outlined in Section 3 in variational formulations suitable for finite element analysis. Define $\mathcal{T}$ and $\mathcal{T}_{-}$ to be discretizations of $\Omega$ and $\Omega_-$, respectively, with $\XOM$ and $\XOMM$ suitable function spaces. In general $\mathcal{T}$ and $\mathcal{T}_-$ have different resolutions and need not be conformal to each other. For ease of exposition, we will refer to the discretization $\mathcal{T}$ of $\Omega$ as `global' and denote its mesh size as $h$, while we will refer to the discretization $\mathcal{T}_-$ of $\Omega_-$ as `local,' denoting its mesh size as $h_-$.
 \par The variational formulation of Algorithm \ref{TwoLevelHeat} is given by:
\begin{alg}\label{TwoLevelHeatVar}
Step 1: Find $u_h^0$ $\in$ $\XTO$ such that for all $v_h$ $\in$ $\XTO$:
\begin{align}\begin{split}\label{pb1Var}
\beta_{+} \int_{\Omega} \nabla u_h^0 \cdot \nabla v_h  &= \int_{\Omega} f \, v_h + \int_{\Gamma_N} q \, v_h \emph{	in	} \Omega\\
u_h^0 &= T_0 \quad \emph{  on   } \Gamma_D
\end{split}\end{align}
Step k: \\

k.1 Find $\widetilde{u}_h^k$ $\in$ $\XTT$ such that for all $w_h$ $\in$ $\XTT$:
\begin{align}\begin{split}\label{pb2Var}
	\beta_- \int_{\Omega_-} \nabla \widetilde{u}_h^k \cdot \nabla w_h &= \int_{\Omega_-} f \, w_h \emph{	in	} \Omega_{-} \\
	\widetilde{u}_h^k &= u_h^{k-1} \emph{	on	} \gamma\end{split}\end{align} 

k.2 Find $\widehat{u}_h^k$ $\in$ $\XTO$ such that for all $v_h$ $\in$ $\XTO$: 
\begin{align}\begin{split}\label{pb3Var}
\beta_{+} \int_{\Omega} \nabla \widehat{u}_h^k \cdot \nabla v_h  &= \int_{\Omega_+} f \, v_h + \frac{\beta_+}{\beta_-}\int_{\Omega_-} f \, v_h + \int_{\Gamma_N} q \, v_h + \int_{\gamma} \left( \beta_+ \frac{\partial \widetilde{u}_h^{k} }{\partial \nn} - \beta_- \frac{\partial \widetilde{u}_h^{k} }{\partial \nn} \right) v_h  \emph{	in	} \Omega \\
\widehat{u}_h^k &= T_0 \emph{ on } \Gamma_D
\end{split}\end{align}

k.3 Relaxation step to obtain \textcolor{black}{the} temperature distribution $u^k$:
\begin{align}\label{relax}
	u_h^k = \theta \widehat{u}_h^k + \left(1-\theta\right)u_h^{k-1}, \quad \theta \in (0, \,1\rbrack.
\end{align}

k.4 Check convergence. If met, terminate iteration. Otherwise repeat step k.	
\end{alg}
The variational formulation of (\ref{TwoLevelHeatUnst}) is similarly given by:
 \begin{alg}\label{TwoLevelHeatVarUnst}
Step 1: Find $u_{h,n}^0$ $\in$ $\XTO$ such that for all $v_h$ $\in$ $\XTO$:
\begin{align}\begin{split}\label{pb1VarU}
\frac{\rho_{+}}{\Delta t} \int_{\Omega} \left(u_{h,n}^0 - u_{h,n-1}\right)  + \beta_{+} \int_{\Omega} \nabla u_{h,n}^0 \cdot \nabla v_h  &= \int_{\Omega} f \, v_h + \int_{\Gamma_N} q \, v_h \emph{	in	} \Omega\\
u_{h,n}^0 &= T_0 \,\, \emph{on } \Gamma_D
\end{split}\end{align}
Step k: \\

k.1 Find $\widetilde{u}_{h,n}^k$ $\in$ $\XTT$ such that for all $w_h$ $\in$ $\XTT$:
\begin{align}\begin{split}\label{pb2VarU}
\frac{\rho_{-}}{\Delta t} \int_{\Omega_-} \left(\widetilde{u}_{h,n}^k - u_{h,n-1}\right) +	\beta_- \int_{\Omega_-} \nabla \widetilde{u}_{h,n}^k \cdot \nabla w_h &= \int_{\Omega_-} f \, w_h \emph{  in  } \Omega_{-} \\
	\widetilde{u}_{h,n}^k &= u_{h,n}^{k-1} \emph{ on } \gamma
	\end{split}\end{align} 

k.2 Find $\widehat{u}_h^k$ $\in$ $\XTO$ such that for all $v_h$ $\in$ $\XTO$: 
\begin{align}\begin{split}\label{pb3VarU}
\MoveEqLeft[1] 
\frac{\rho_+}{\Delta t} \int_{\Omega} \left(\widehat{u}_{h,n}^k  - u_{h,n-1} \right)v_h +   \beta_{+} \int_{\Omega} \nabla \widehat{u}_{h,n}^k \cdot \nabla v_h \\ \quad{}&  = \int_{\Omega_+} f \, v_h + \frac{\beta_+}{\beta_-}\int_{\Omega_-} f \, v_h + \int_{\Gamma_N} q \, v_h + \int_{\gamma} \left( \beta_+ \frac{\partial \widetilde{u}_{h,n}^{k} }{\partial \nn} - \beta_- \frac{\partial \widetilde{u}_{h,n}^{k} }{\partial \nn} \right) v_h  \\ \quad{}& + \left(\rho_+ -\frac{\rho_-\beta_+}{\beta_-}\right) \frac{1}{\Delta t} \int_{\Omega_-} \left(\widetilde{u}_{h,n}^k  - u_{h,n-1} \right)v_h \emph{	in	} \Omega \\
\widehat{u}_{h,n}^k &= T_0 \emph{ on } \Gamma_D
\end{split}\end{align}

k.3 Relaxation step to obtain final temperature distribution $u_n^k$:
\begin{align}\label{relax}
	u_{h,n}^k = \theta \widehat{u}_{h,n}^k + \left(1-\theta\right)u_{h,n}^{k-1}, \quad \theta \in (0, \,1\rbrack.
\end{align}

k.4 Check convergence. If met, terminate iteration. Otherwise repeat step k.	
\end{alg}

\par \textbf{Remark: Note on boundary conditions. } We briefly address a practical concern regarding the implementation of Neumann boundary conditions when using the Algorithms \ref{TwoLevelHeatVar} and \ref{TwoLevelHeatVarUnst}. For problems where $\Gamma_N \cap \gamma \neq \emptyset $, at first glance the Neumann boundary conditions \textcolor{black}{in} (\ref{pb3Var}) and (\ref{pb3VarU}) may appear incorrect as we have:
\begin{align}\label{BCConundrum1}
\beta_{+} \frac{\partial u}{\partial \nn} &= q \text{  on  } \Gamma_N
\end{align}
while the corresponding condition \textcolor{black}{in} (\ref{Model1}) is:
\begin{align}\label{BCConundrum2}
\beta_{-} \frac{\partial u}{\partial \nn} &= q \text{  on  } \Gamma_N.
\end{align}
\par One recalls quickly that for Problem (\ref{Model1})\footnote{The same argument holds for (\ref{Model2}) without any loss of generality.}, a Neumann boundary condition is enforced naturally by the variational formulation. For generic function space $V$, after integration by parts the variational formulation of (\ref{Model1}) is given by: Find $u \in V$ such that for all $\varphi \in V$:
\begin{align}\begin{split}\label{BCConundrum4}
\beta \int_{\Omega}  \nabla u \cdot \nabla \varphi - \beta \int_{\Gamma_N}  \frac{\partial u}{\partial \nn} \varphi &= \int_{\Omega} f \varphi \emph{	in	} \Omega \\
u &= T_0 \text{ on } \Gamma_D.
\end{split}\end{align}
The \textcolor{black}{Neumann boundary} condition:
 \textcolor{black}{\begin{align}\label{NeumanBC}\beta\, \frac{\partial u}{\partial \nn} = q \end{align}}is then enforced by simply replacing \textcolor{black}{$\beta\displaystyle\frac{\partial u}{\partial \nn}$} in (\ref{BCConundrum4}) with $q$, giving the problem:
Find $u \in V$ such that for all $\varphi \in V$:
\begin{align}\begin{split}\label{BCConundrum5}
\beta \int_{\Omega}  \nabla u \cdot \nabla   \varphi -  \int_{\Gamma_N} q \, \varphi &= \int_{\Omega} f \varphi  \\
u &= T_0 \text{ on } \Gamma_D \end{split}\end{align}
 We now turn our attention to (\ref{pb3Var}) and (\ref{pb3VarU}). Let $\widetilde{\gamma} = \gamma\backslash\Gamma_N$, hence $\gamma = \widetilde{\gamma} \cup \Gamma_N$. On the right hand side of (\ref{pb3Var}) one then has the boundary integrals (the same argument also applies for (\ref{pb3VarU})):
\begin{align}\begin{split}\label{BCConundrum6}
	\int_{\Gamma_N} q v_h + \beta_+ \int_{\Gamma_N} \frac{\partial \widetilde{u}_h^k}{\partial \nn}v_h - \beta_- \int_{\Gamma_N} \frac{\partial \widetilde{u}_h^k}{\partial \nn}v_h + \beta_+ \int_{\widetilde{\gamma}} \frac{\partial \widetilde{u}_h^k}{\partial \nn}v_h - \beta_- \int_{\widetilde{\gamma}} \frac{\partial \widetilde{u}_h^k}{\partial \nn}v_h
\end{split}\end{align}
By the same reasoning as in (\ref{BCConundrum5}), \textcolor{black}{we enforce} the condition:
\begin{align}\begin{split}\label{BCConundrum7}
-\int_{\Gamma_N} \beta_+ \frac{\partial u_h^k}{\partial \nn} v_h  &= -\int_{\Gamma_N} \beta_+ \frac{\partial \widetilde{u}_h^k}{\partial \nn} v_h +\int_{\Gamma_N} \beta_- \frac{\partial \widetilde{u}_h^k}{\partial \nn} v_h	- \int_{\Gamma_N} q v_h
\end{split}\end{align}
As $\widetilde{u}_h^k = u_h^{k-1}$ on $\gamma$, for a convergent scheme and $k$ sufficiently large, one has $\widetilde{u}_h^k \approx u_h^k$, and (\ref{BCConundrum7}) reduces to :
\begin{align}\begin{split}\label{BCConundrum8}
- \int_{\Gamma_N} \beta_- \frac{\partial \widetilde{u}_h^k}{\partial \nn} v_h  &= 	- \int_{\Gamma_N} q v_h,
\end{split}\end{align}
 consistent with (\ref{BCConundrum2})\footnote{We acknowledge an abuse of notation, as at the continuous level all integrals over $\Gamma_N$ are zero since $m\left(\Gamma_N\right)=0$. However, as this is not the case at the discrete level, and as this is of significant practical interest for those wishing to implement the outlined methods, we feel this presentation is justified in the interest of clarity.}.

\section{Numerical Examples}
We now seek to demonstrate the applicability of the discussed methods with a pair of two-dimensional numerical problems designed to test different aspects of the proposed approach. For both problems, we will refer to comparison solutions computed by solving the equations in a standard (i.e., no two-level treatment) way on a mesh $\mathcal{T}$ as \textit{monolithic} solutions. 
\par We investigate two problems:

\begin{enumerate}
\item \textbf{Steady Additive-type Problem.} \textcolor{black}{A steady-state problem inspired by additive manufacturing for which $\displaystyle{\gamma \cap \partial  \Omega \neq \emptyset}$. In this test we investigate the relationship between the global mesh size $h$ and the local mesh size $h_-$. In particular, for  fixed values of $h$, we seek to observe the effect of refining $h_-$ on the error behavior.}
\item \textbf{Unsteady Additive-type Problem.} \textcolor{black}{An unsteady problem adopting the same geometrical configuration as the previous test. Again, we will test several different refinement levels of $h_-$ for given $h$. We are interested in observing the behavior of the method for unsteady problems in terms of both accuracy and temporal stability.} 
 
 \end{enumerate}
\subsection{Steady Additive-type Problem}
In this example we solve a steady thermal problem with the configuration  shown in Fig. \ref{fig:TestOneGeo}. We set $H=1.0$, $L=1.0$, $H_-=.05$, $\beta_+=1.0$, $\beta_-=20.0$, $T_0=20$, and $q$ defined as:
$$ q = 2000 \text{ exp }\left(-\frac{(.1-x)^2}{.0004} \right). $$ 
\par  The investigated geometry is common in additive manufacturing, where the physics within the thin topmost layer of the problem ($\Omega_-$) is often different than in the remainder of the domain. We note that such problems from are an important potential application for the proposed method, since additive manufacturing is an active area of research \cite{BRDP2018, GMWP2012, HYDY2016,IM2016, KAFHKKR2015, KOCZR2018, PPRZMHBS2015, PPRZMHBS2016, RSM2014,  WLGNMR2017}. 
\par  We seek to observe the relationship between the error, the global mesh level $h$, and the local mesh level $h_-$. In particular, we are interested in observing how the refinement of $h_-$ for a given level of $h$ affects the error behavior. This test is distinct from the one that will follow as it is steady problem; hence we will not consider any temporal effects and their impact on the methodology or the computed solution.
\par We first compute a reference solution $u_{ref}$ on a fine uniform mesh with $h = 1/500$. \textcolor{black}{We compare with computed solutions on three global mesh levels: $h=1/20,\,1/40,$ and $1/80$. For the $h=1/20$ and $h=1/40$ cases, we then compute monolithic solutions without the two-level method treatment as well as solutions using the two-level method for $h_-=1/80,\,/160$, and $1/240$. For the case $h=1/80$, we compute two-level solutions only for $h_-=1/160$ and $h_-=1/240$.} We used $\mathbb{P}^2$ piecewise quadratic finite elements for all simulations.
\begin{figure}
\centering
\includegraphics[scale=.75]{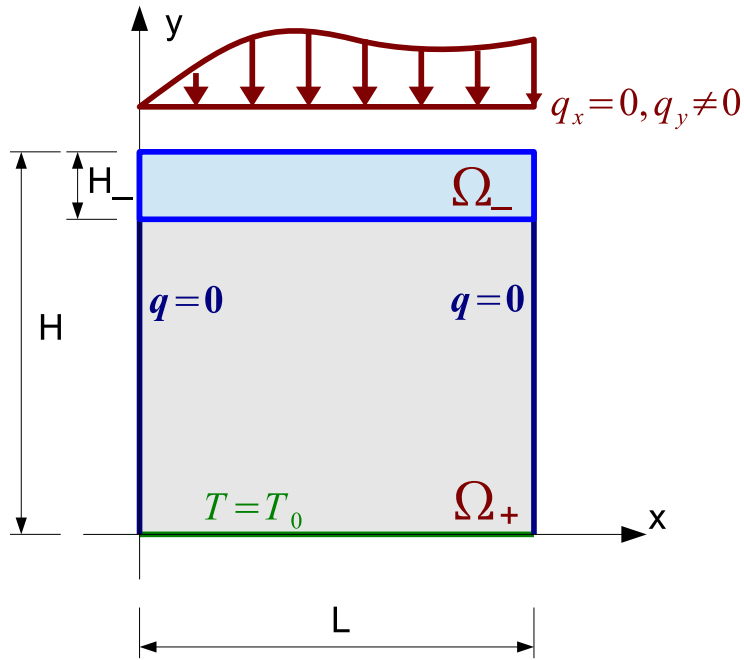}\caption{Steady Additive-type Problem configuration. $\Omega$ consists of a rectangular domain of length $L$ and height $H$, with $\Omega_-$ defined as a rectangular region along the top of $\Omega$ with length $L$ and height $H_-$. A source term $q$ is given along the upper border of $\Omega$ (which is shared by $\Omega_-$), with a fixed temperature $T_0$ prescribed along the lower border. The material properties differ in $\Omega_-$ and $\Omega_+$.}\label{fig:TestOneGeo}
\end{figure}
\par We plot the error compared to the reference solution for each configuration in Fig. \ref{fig:TestOneError}. \textcolor{black}{Note that for each curve, the rightmost point simply corresponds to the solution obtained without using the two-level algorithm.} The plotted figure shows that continued refinement of $h_-$ leads to improved error behavior for each level of $h$. The test demonstrates that one only need refine the local mesh $\mathcal{T}_-$ in order to gain accuracy using this method. Additionally, the test further supports the findings of the previous one, as we again observe that the accuracy improvements do not become less pronounced as we refine the global mesh level.      
\begin{figure}
\centering
\includegraphics[scale=.325]{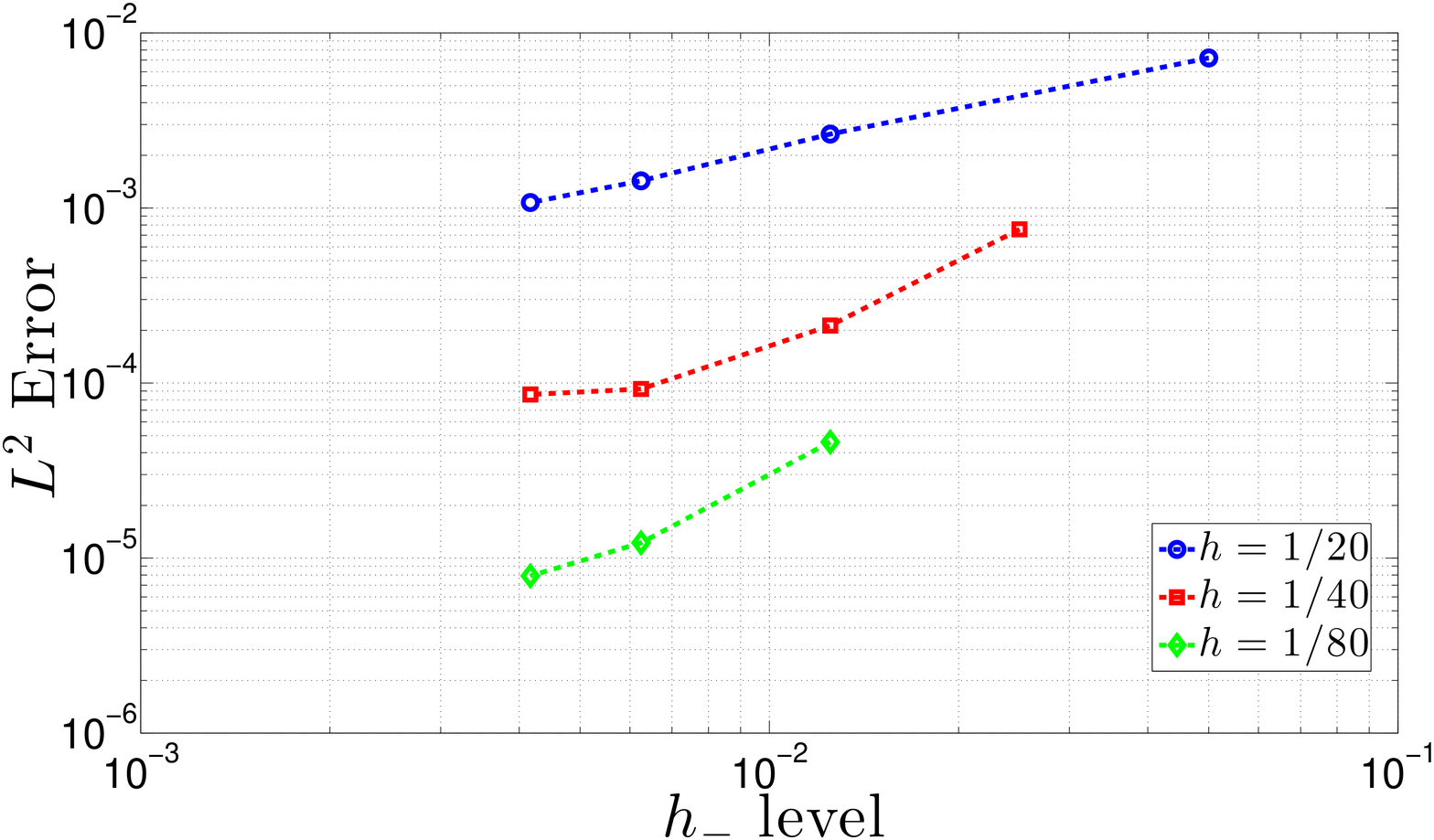}\caption{Steady Additive-type Problem; Relative error in $L^2$ norm as with a reference solution compared to the local mesh size $h_-$. The different lines correspond to different global mesh levels $h$, with the rightmost point on each line representing the monolithic solution. For each global mesh level $h$, we examine the error behavior as we refine the local mesh level $h_-$.}\label{fig:TestOneError}
\end{figure}

\subsection{Unsteady Additive-type Problem}

\par In this example, we apply the time-dependent Algorithm \ref{TwoLevelHeatVarUnst} on a problem inspired by additive manufacturing using a configuration depicted in Fig. \ref{fig:TestThreeGeo}. The problem is designed to model the heating and cooling cycles of a material during the additive process. During each cycle, indicated with the letter $i$, the material undergoes a heating phase in which a laser heats the top layer left-to-right, followed by a cooling phase. \textcolor{black}{Problems in additive manufacturing are of particular interest for the proposed method,} as along the upper portion the physics are generally more involved and require extra attention. This test is similar in its geometry to the previous test; however it is quite different from the physical point of view, as we are now considering an unsteady problem. We seek to analyze the stability and accuracy of the proposed approach when applied to the unsteady setting, and demonstrate its ability to properly resolve temporal effects.

\begin{figure}
\centering
\includegraphics[scale=.75]{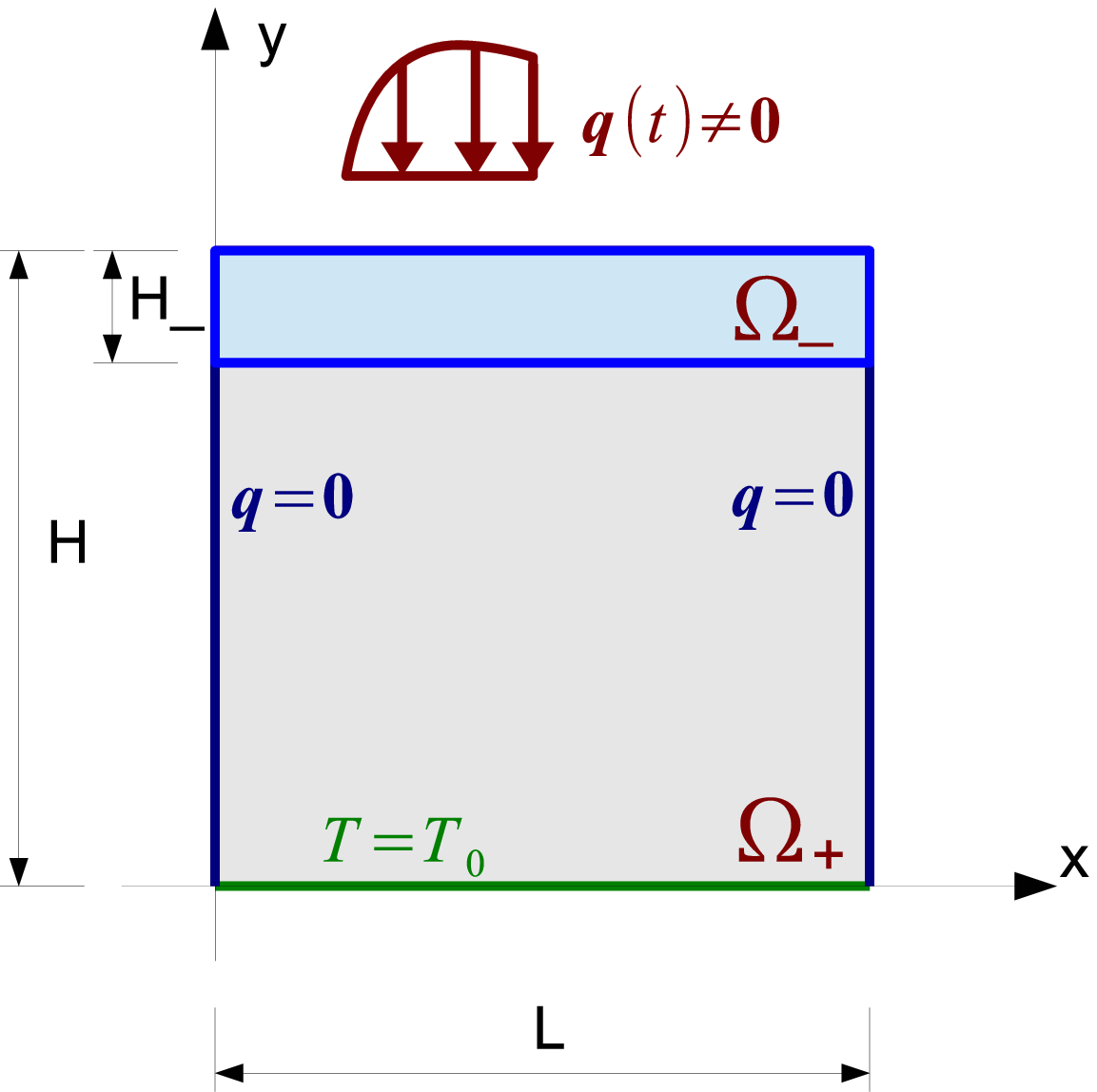}\caption{Unsteady Additive-type Problem configuration. $\Omega$ consists of a rectangular domain of length $L$ and height $H$, with $\Omega_-$ defined as a rectangular region along the top of $\Omega$ with length $L$ and height $H_-$. A time-dependent source term $q(t)$ moves from left-to-right along the upper border of $\Omega$ (which is shared by $\Omega_-$), with a fixed temperature $T_0$ prescribed along the lower border. The material properties differ in $\Omega_-$ and $\Omega_+$.}\label{fig:TestThreeGeo}
\end{figure}

\par  For a heating/cooling cycle $i$, we denote the beginning of the cycle (and hence the heating phase) as $t_{i,0}$, the beginning of the cooling phase as $t_{i,\,cool}$ and the end of the cycle as $t_{i,\,end}$. Using this notation, we define the time-dependent source term as:
\begin{align}\label{UnsteadySource2}
q(x,t,t_{h,0}) &= \begin{cases} 2000\text{ exp }\left(-\frac{ \left(10\left(t-t_{h,0}\right)-x\right)^2 }{.0005}\right) \,\,\,\, \text{for } t \in \, \lbrack t_{i,0},\,t_{i,cool} \rbrack \\
0 \,\,\,\, \text{for } t \in \, \lbrack t_{i,cool},\,t_{i,end} \rbrack. \\
\end{cases}\end{align}
The proposed time-dependent heating source is designed to model the action of a laser heating the top material from left-to-right over a length of $.1$ seconds during a heating cycle, with the cooling cycle lasting $.07$ seconds. This process is illustrated in Fig. \ref{fig:AddMan}.
\par We set $T_0=20$, $\beta_+=1.0$, $\beta_-=20.0$, $\rho_+=1.0$, $\rho_-=5.0$, and $\Delta t=.01$. Our reference solution is computed on uniform fine mesh with $h =1/500$. We then compute four solutions on a uniform global mesh with $h = 1/20$: a monolithic solution, and two-level solutions with $h_-=1/80,\,1/100,\,1/120$. For all simulations, we discretize using $\mathbb{P}^2$ piecewise polynomial finite elements and simulate five full cycles. 

\begin{figure}

\includegraphics[width=\linewidth]{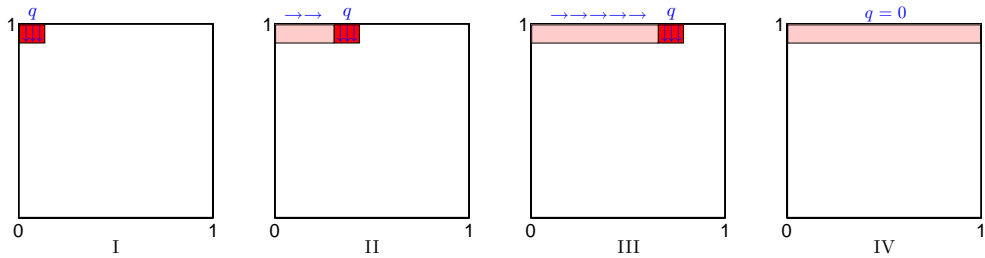}	\caption{Unsteady Additive-type Problem, heating cycle process (left-to-right). In each cycle, the laser then heats the material left-to right (I-III). The laser is then switched off and the material cools (IV).}\label{fig:AddMan}
\end{figure}
We plot two relevant results. Figure \ref{fig:UnsteadyTemp} shows the temperature at different heights in time along the vertical line $x=.5$, showing agreement between the two-level ($h_-=1/120$) and reference solutions. We have colored the plot to indicate the heating and cooling phases. Figure \ref{fig:UnsteadyTempError} shows the relative $L^2$ error in time for the different solutions as compared to the reference. We observe that, although the two-level method does not give significant accuracy increases for $h_-=1/80$, the performance improves noticeably as we refine $h_-$. For $h_-=1/100$ and particularly $h_-=1/120$, we significantly reduce the error when compared to the standard solution. We again emphasize that the same global mesh with $h=1/20$ was used for all simulations, and that the observed improvements result solely from refining the local mesh.

\begin{figure}

\centering

\begin{subfigure}[h]{1.0\linewidth}
\includegraphics[width=\linewidth]{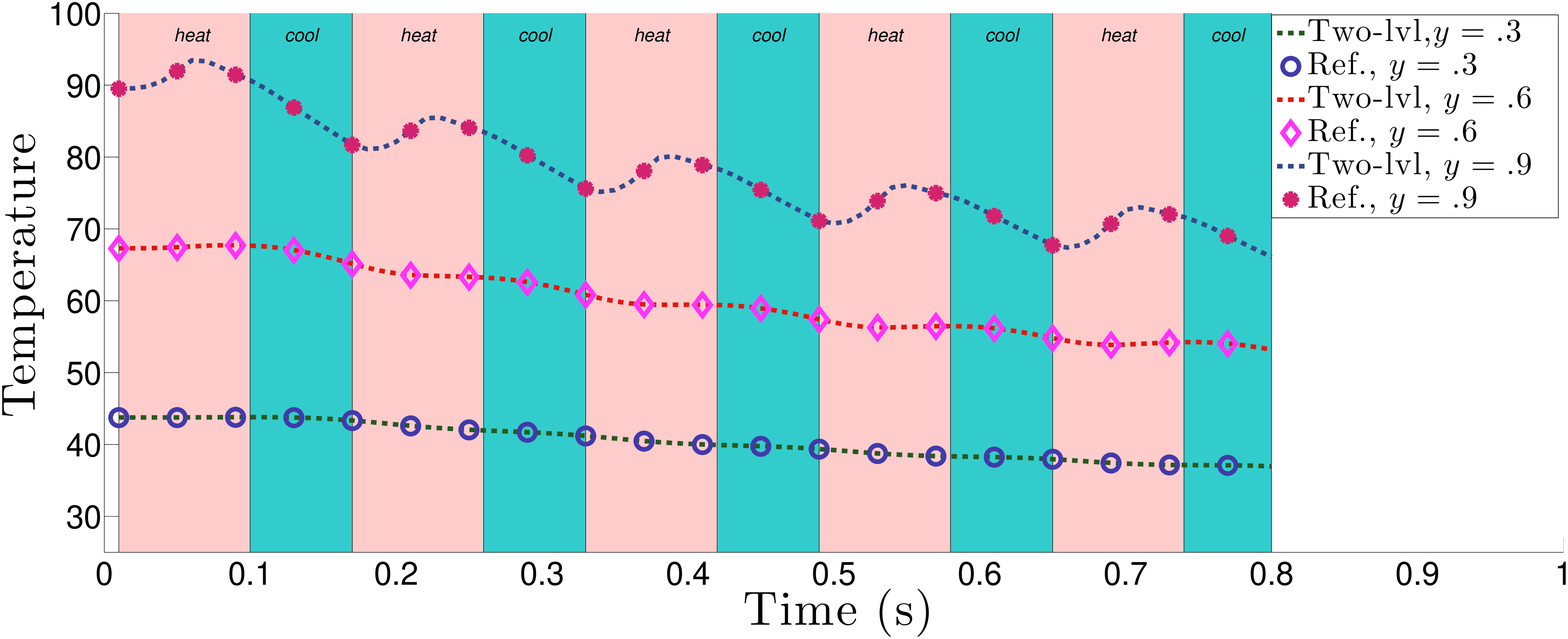}\end{subfigure}\newline
\begin{subfigure}[h]{.4\linewidth}
\includegraphics[width=\linewidth]{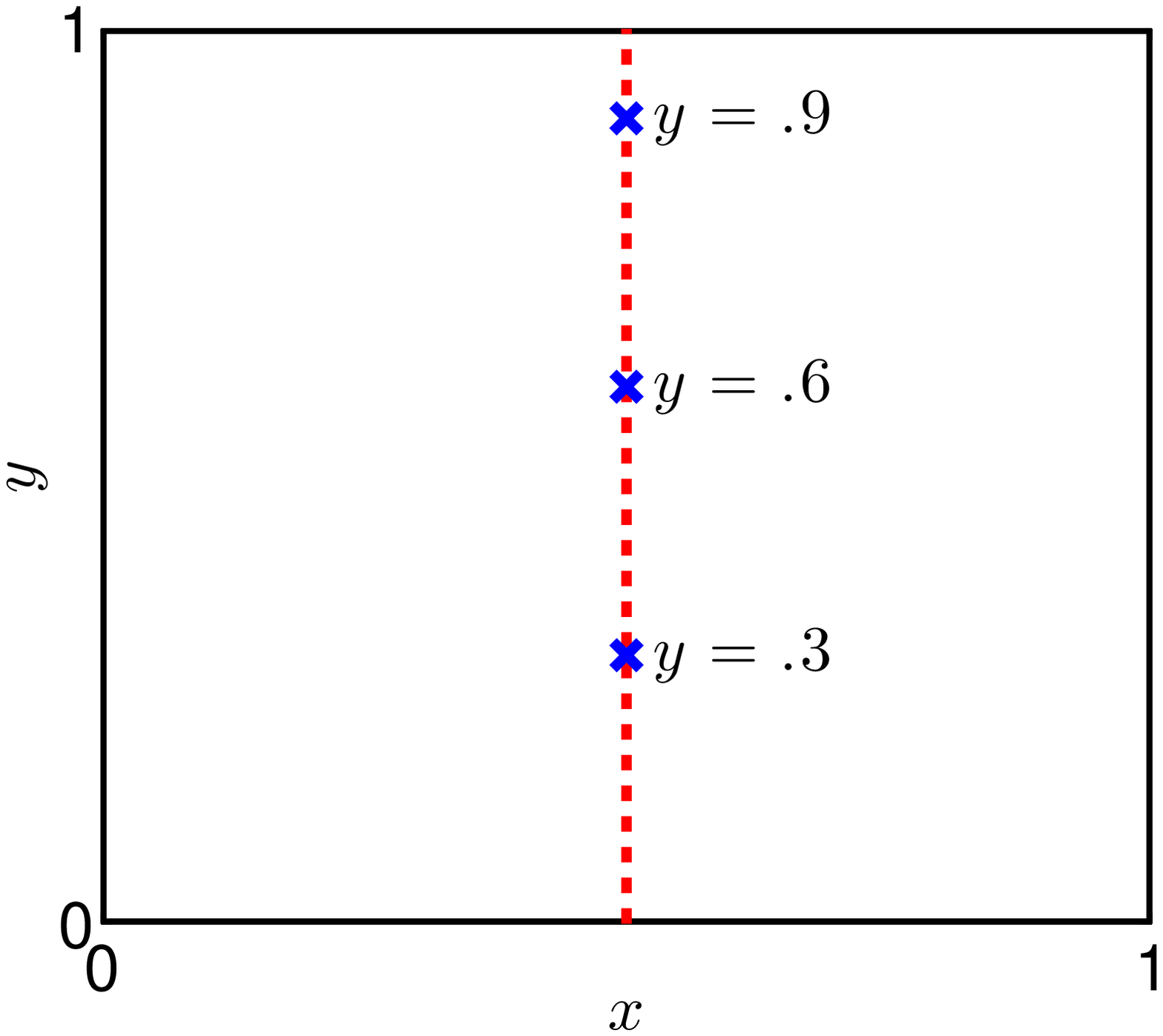}
\end{subfigure}
\caption{Unsteady Additive-type Problem; the temperature in time of the two-level solution (the lines) compared to the reference solution (marker points) evaluated at different points in the domain (bottom).}\label{fig:UnsteadyTemp}
\end{figure}

\begin{figure}
\centering
\includegraphics[scale=.3]{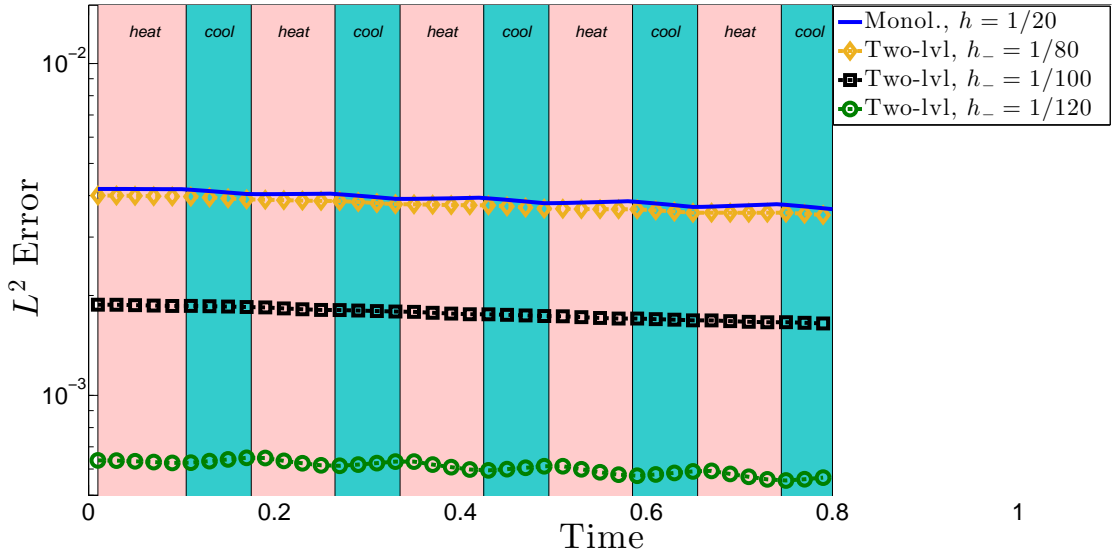}\caption{Unsteady Additive-type Problem; Relative error in $L^2$ norm as compared to a reference solution in time. The global mesh level is the same for each case, with the different lines representing the monolithic solution (blue) and different levels of $h_-$.}\label{fig:UnsteadyTempError}
\end{figure}

\section{Conclusions}
We have introduced a two-level approach for solving partial differential equations, related to the previously introduced Fat boundary method. This method is based on decomposing the problem into a split-problem formulation, then solving the two problems on different meshes in an iterative manner until convergence is achieved. It has shown particular promise for problems with prominent local characteristics, where some phenomena only occur on a (potentially small) portion of the mesh. Problems of this type occur in many areas of industrial interest, including additive manufacturing.
\par The proposed method contains many important and nontrivial mathematical concerns. While we have proved consistency with the original problem, a formal proof of convergence remains missing. In a similar vein, the dependance of the convergence on various parameters, including the under-relaxation parameter $\theta$ and for unsteady problems, the time step $\Delta t$ and temporal order of convergence, should be investigated. The algorithm described here also has important implications regarding the construction and application of preconditioners for the associated linear systems. We expect the outlined approach to be beneficial in this respect, as the split-problem formulation allows one to employ uniform (or quasi-uniform) meshes for the separate problems, mitigating much of the difficulty associated with preconditioning. However, this must be explored and confirmed in more detail.
\par For ease of presentation and mathematical analysis, we have restricted our attention in this work to cases in which $\beta_-,\,\beta_+,\,\rho_-,$ and $\rho_+$ do not depend on the unknown $u$. For nonlinear problems where some or all of these parameters may depend on $u$, one must extend the approaches shown here. A natural example of such an extension is the Picard-type iterative method obtained by employing analogous two-level methods based on the splittings (\ref{pb1NoAlgNonCons})-(\ref{pb2NoAlgNonConsUnst}), and at an iteration $k$, letting  $\beta\left(u\right):=\beta\left(u^{k-1}\right)$ and $\rho\left(u\right):=\rho\left(u^{k-1}\right)$. While some initial tests (not presented here) suggest this approach works, such a scheme must be analyzed and tested more rigorously. The development of a Newton-type scheme, where the nonlinear solver may converge more rapidly, is also worthy of further investigation.
\par From the engineering and industrial perspective, the proposed method must be further validated on more realistic problems. In particular, we intend to apply it to additive manufacturing problems of genuine engineering interest, in which we incorporate phase change and other such phenomena, in both two and three dimensions. As these problems are highly nonlinear in general, the concerns regarding the algorithm's application to nonlinear problems is especially important.

\section{Appendix: Proofs of Theorems \ref{thm:nonconsSteady} and \ref{thm:nonconsUnsteady}}
In this section, we provide proofs for Theorems \ref{thm:nonconsSteady} and \ref{thm:nonconsUnsteady}. We note the extension of the splittings given by Problems \ref{pb1NoAlgNonCons} and \ref{pb2NoAlgNonCons} (steady) and Problems \ref{pb1NoAlgNonConsUnst} and \ref{pb2NoAlgNonConsUnst} (unsteady) to algorithms of the type shown in Section 3 then follow in the obvious way. 
\par For fully nonlinear problems in which $\beta$ and/or $\rho$ are not simply non-constant but depend on the unknown temperature field $u$, the following theorems are useful to establish the consistency of a single-step of some nonlinear iterative scheme (such as a Picard-type method). Proving the convergence of such a scheme, however, is beyond the scope of this work and an important subject for future research.
\subsection{Proof of Theorem \ref{thm:nonconsSteady}. }
\begin{proof} Let $u$ be a solution of (\ref{Model1}) where $\beta$ is understood to be non-constant. Define $\widetilde{u}$ as:
\begin{eqnarray}\label{uTildeDefNL}
\widetilde{u} := u|_{\Omega_-}.
\end{eqnarray}
(\ref{uTildeDefNL}) ensures $\widetilde{u}$ then satisfies (\ref{pb2NoAlgNonCons}) trivially by definition.  Then for $u \in H^{1}\left(\Omega\right)$ and $\varphi \in C_{0}^{\infty}\left(\Omega\right)$:
\begin{align}\begin{split}\label{proofStep2NL}
\int_{\Omega} \beta \nabla{u} \cdot \nabla{\varphi} &= \int_{\Omega} \beta_+ \nabla{u} \cdot \nabla{\varphi} - \int_{\Omega_-} \beta_+ \nabla u \cdot \nabla \varphi + \int_{\Omega_-} \beta_- \nabla u \cdot \nabla \varphi  \\
	&= \int_{\Omega} \beta_+ \nabla{u} \cdot \nabla{\varphi} - \int_{\Omega_-} \beta_+ \nabla \widetilde{u} \cdot \nabla \varphi + \int_{\Omega_-} \beta_- \nabla \widetilde{u} \cdot \nabla \varphi.  
\end{split}\end{align}
We now prove a pair of brief lemmas which we will need to proceed.
\begin{lemma}\label{omegaPlusLabelNL}
Under the assumptions of the theorem, if $\beta_-\neq 0$ in $\Omega_-$,	\begin{align}\label{betaPlusIdentity} -\nabla \cdot\left( \beta_+ \nabla \widetilde{u}\right) = \frac{\beta_+}{\beta_-}f + \frac{\beta_+}{\beta_-}\nabla \widetilde{u} \cdot \nabla \beta_- - \nabla \widetilde{u}\cdot\nabla\beta_+\,\,\, \text{ in } \Omega_-. \end{align}
\end{lemma}
\begin{proof}
Since $u$ satisfies (\ref{Model1}) weakly in $\Omega$, it follows from (\ref{proofStep2NL}) that $$\int_{\Omega} \beta_+ \nabla{u} \cdot \nabla{\varphi} - \int_{\Omega_-} \beta_+ \nabla \widetilde{u} \cdot \nabla \varphi + \int_{\Omega_-} \beta_- \nabla \widetilde{u} \cdot \nabla \varphi =\int_{\Omega} f\varphi\,\,\, \text{ in } \Omega_-.$$ In particular, we have that:
\begin{align}\label{betaMinusIdentity}-\nabla \cdot\left( \beta_- \nabla \widetilde{u}\right)  = f \,\,\, \text{ in } \Omega_-. \end{align}
We recall the product rule for the divergence of a product of a vector function $\aaa$ and a scalar function $b$: 
\begin{align}\label{divergenceProductRule}
	\nabla \cdot \left(b \aaa\right) &= b \Delta \aaa + \nabla b \cdot \aaa
\end{align}
Applying (\ref{divergenceProductRule}) and elementary manipulations to (\ref{betaMinusIdentity}) gives:
\begin{align}\begin{split}\label{lemmaStep1}
-\beta_+ \Delta \tu &= \frac{\beta_+}{\beta_-} f + \frac{\beta_+}{\beta_-} \nabla \tu \cdot \nabla \beta_-  \end{split}\end{align}
Adding and subtracting $\nabla \tu \cdot \nabla \beta_+ $ to (\ref{lemmaStep1}), we then have:
\begin{align}\begin{split}\label{lemmaStep2}
-\beta_+ \Delta \tu -\nabla \tu \cdot \nabla \beta_+ + \nabla \tu \cdot \nabla \beta_+   &= \frac{\beta_+}{\beta_-} f + \frac{\beta_+}{\beta_-} \nabla \tu \cdot \nabla \beta_- \\
-\nabla \cdot \left(\beta_+ \nabla \tu\right) &= \frac{\beta_+}{\beta_-} f + \frac{\beta_+}{\beta_-} \nabla \tu \cdot \nabla \beta_- - \nabla \tu \cdot \nabla \beta_+
\end{split}\end{align}
where the last line follows from applying (\ref{divergenceProductRule}) in reverse, completing the proof.
\end{proof}
\begin{lemma}\label{omegaPlusLemma2}
Under the assumptions of the theorem, if $\beta_-\neq 0$ in $\Omega_-$,	\begin{align}\label{betaLaplaceIdentity} 
	\beta_+ \Delta \tu &= \frac{\beta_+}{\beta_-} \big\lbrack  \nabla\cdot\left(\beta_- \nabla \tu\right) - \nabla \beta_- \cdot \nabla \tu \big\rbrack \end{align}
\end{lemma}
\begin{proof}
From (\ref{divergenceProductRule}):
\begin{align}\begin{split}\label{lemma2Step1}
	\nabla \cdot \left(\beta_+ \nabla\tu\right) &= \nabla\beta_+ \cdot \nabla \tu + \beta_+ \Delta \tu  \\
	&= \nabla\beta_+ \cdot \nabla \tu + \frac{\beta_+}{\beta_-}\beta_- \Delta \tu \\
	&= \nabla\beta_+ \cdot \nabla \tu + \frac{\beta_+}{\beta_-} \bigg\lbrack \nabla \cdot \left(\beta_- \nabla \tu\right) - \nabla \beta_- \cdot \nabla \tu \bigg\rbrack,
\end{split}\end{align}
where the last line follows from a second application of (\ref{divergenceProductRule}). This implies:
\begin{align}\begin{split}
	\nabla\beta_+ \cdot \nabla \tu + \beta_+ \Delta \tu &= \nabla\beta_+ \cdot \nabla \tu + \frac{\beta_+}{\beta_-} \bigg\lbrack \nabla \cdot \left(\beta_- \nabla \tu\right) - \nabla \beta_- \cdot \nabla \tu \bigg\rbrack \\
	 \beta_+ \Delta \tu &=  \frac{\beta_+}{\beta_-} \bigg\lbrack \nabla \cdot \left(\beta_- \nabla \tu\right) - \nabla \beta_- \cdot \nabla \tu \bigg\rbrack,
\end{split}\end{align}
establishing the lemma.
\end{proof}
Returning to the main theorem, integrating by parts the integrals in $\Omega_{-}$ on the right hand side of (\ref{proofStep2NL}) yields:
\begin{align}\begin{split}\label{proofStep3NL}\MoveEqLeft[1] 
- \int_{\Omega_-} \beta_+ \nabla \widetilde{u} \cdot \nabla \varphi + \int_{\Omega_-} \beta_- \nabla \widetilde{u} \cdot \nabla \varphi \\ \quad{}& = \int_{\Omega_-} \nabla \cdot\left( \beta_+ \nabla \widetilde{u}\right) \varphi - \int_{\gamma} \beta_{+} \frac{\partial \widetilde{u}}{\partial \nn} \varphi - \int_{\Omega_-} \nabla \cdot\left( \beta_- \nabla \widetilde{u}\right)  \varphi + \int_{\gamma} \beta_{-} \frac{\partial \widetilde{u}}{\partial \nn} \varphi \\
\end{split}\end{align}	
Which from (\ref{betaPlusIdentity}) and (\ref{betaMinusIdentity}) reduces to:
\begin{align}\begin{split}\label{proofStep4NL}\MoveEqLeft[2] 
- \int_{\Omega_-} \beta_+ \nabla \widetilde{u} \cdot \nabla \varphi + \int_{\Omega_-} \beta_- \nabla \widetilde{u} \cdot \nabla \varphi \\ \quad{}& =  \int_{\Omega_-} \left(1-\frac{\beta+}{\beta_-}\right) f\varphi - \int_{\Omega_-} \frac{\beta+}{\beta_-} \left(\nabla \tu \cdot \nabla\beta_-\right) \varphi + \int_{\Omega_-} \left(\nabla \tu \cdot \nabla \beta_+\right) \varphi + 
\int_{\gamma} \left(\beta_- - \beta_+\right)\frac{\partial \widetilde{u}}{\partial \nn} \varphi 
\\
\end{split}\end{align}	
Combining (\ref{proofStep4NL}) with (\ref{proofStep2NL}) then gives: 
\begin{align}\begin{split}\MoveEqLeft[2]
    \int_{\Omega} \beta_+ \nabla{u} \cdot \nabla{\varphi} \\ \quad{}& = \int_{\Omega_+} f\varphi +   \int_{\Omega_-} \frac{\beta+}{\beta_-} f\varphi  
     + \int_{\Omega_-} \frac{\beta+}{\beta_-} \left(\nabla \tu \cdot \nabla\beta_-\right) \varphi
     \\ \quad{}&
     - \int_{\Omega_-} \left(\nabla \tu \cdot \nabla \beta_+\right) \varphi +
    \int_{\gamma} \beta_+ \frac{\partial \widetilde{u} }{\partial \nn} \varphi -  \int_{\gamma}\beta_-\frac{\partial \widetilde{u} }{\partial \nn} \varphi,
\end{split}\end{align}
which after backward integration over $\Omega$ gives:
\begin{align}\begin{split}\label{resultSteadyNL}
\MoveEqLeft[1]
  	-\nabla \cdot \left(\beta_{+} \nabla {u} \right) \\ \quad{}& =  f\big|_{\Omega_+} +\frac{\beta_+}{\beta_-}f\big|_{\Omega_-}
  	+\frac{\beta_+}{\beta_-}\nabla \widetilde{u} \cdot \nabla \beta_- \big|_{\Omega_-}
  	- \nabla \widetilde{u} \cdot \nabla \beta_+ \big|_{\Omega_-}
  	    +  \left(\beta_{+}-\beta_-\right) \frac{\partial \widetilde{u} }{\partial \nn}  \delta_{\gamma}   \text{  in  } \Omega.
\end{split}\end{align}
which was to be shown.
\par Now let $(u,\tu)$ be a solution pair of the coupled problem. We seek to verify that $u=\tu$ in $\Omega_-$. This is equivalent to establishing that $u-\tu$ satisfies the equation:
\begin{align}\label{verifPbm}
-\nabla \cdot \beta_+ \nabla (u - \widetilde u) = 0, \quad \text{in }\Omega_-, \qquad u - \widetilde u = 0, \quad \text{in }\gamma,
\end{align} 
which has unique solution zero. Note that $u=\tu$ on $\gamma$ by hypothesis. For all $\varphi \in H_0^1(\Omega_-)$, we have\footnote{Note that as (\ref{verifPbm}) is zero on the boundary, we may demand this on our test space as well without loss of generality.}:
\begin{align}
\int_{\Omega_-} \beta_+ \nabla\left(\tu-u\right)\cdot \nabla \varphi	  &= - \int_{\Omega_-} \nabla \cdot \left\lbrack \beta_+ \nabla\left(\tu-u\right)\right\rbrack\varphi \\
&= -\int_{\Omega_-} \nabla\cdot \left( \beta_+ \nabla \tu \right) \varphi + \int_{\Omega_-} \nabla\cdot \left( \beta_+ \nabla u \right) \varphi \\
\begin{split}\label{SteadyNonConsStep2}
&= -\int_{\Omega_-} \nabla \beta_+ \cdot \nabla\tu \, \varphi - \int_{\Omega_-} \beta_+ \Delta \tu \, \varphi + \int_{\Omega_-} \nabla\cdot \left( \beta_+ \nabla u \right) \varphi. \\
\end{split}\end{align}
By applying Lemma \ref{omegaPlusLemma2} to the middle term of (\ref{SteadyNonConsStep2}), we then obtain:
\begin{align}\begin{split}\label{SteadyNonConsStep3}\MoveEqLeft[1]
\int_{\Omega_-} \beta_+ \nabla\left(\tu-u\right)\cdot \nabla \varphi \\
&{}= -\int_{\Omega_-} \nabla \beta_+ \cdot \nabla\tu \, \varphi -\int_{\Omega_-}  \nabla \cdot \left(\beta_- \nabla\tu\right)\frac{\beta_+}{\beta_-} \varphi + \int_{\Omega_-} \frac{\beta_+}{\beta_-} \nabla \beta_-\cdot \nabla \tu + \int_{\Omega_-} \nabla\cdot \left( \beta_+ \nabla u \right) \varphi.
\end{split}\end{align}
Applying integration by parts to the second term on the right-hand side of (\ref{SteadyNonConsStep3}), we observe that:
\begin{align}\begin{split}\label{SteadyNonConsStep4}
	-\int_{\Omega_-}  \nabla \cdot \left(\beta_- \nabla\tu\right)\frac{\beta_+}{\beta_-} \varphi &= \int_{\Omega_-} \beta_- \nabla \tu \cdot \nabla\left(\frac{\beta_+}{\beta_-} \varphi\right),
\end{split}\end{align}
with $\varphi\in H_0^1\left(\Omega_-\right)$ ensuring zero boundary terms. As $\tu$ satisfies (\ref{pb2NoAlgNonCons}) weakly in $\Omega_-$, this implies that for all $\varphi \in H^1(\Omega_-)$:
\begin{align}\label{weakSatisf}
	\int_{\Omega_-} \beta_- \nabla \tu \cdot \nabla \varphi &= \int_{\Omega_-} f\,\varphi.
\end{align}
As $\beta_+/\beta_- \, \in H^1(\Omega_-)$, so too is $\frac{\beta_+}{\beta_-}\varphi \, \in H_0^1(\Omega_-)$ and this together with (\ref{SteadyNonConsStep4}) and (\ref{weakSatisf}) yields:
\begin{align}\begin{split}\label{SteadyNonConsStep5}
	-\int_{\Omega_-}  \nabla \cdot \left(\beta_- \nabla\tu\right)\frac{\beta_+}{\beta_-} \varphi &= \int_{\Omega_-} f\,\frac{\beta_+}{\beta_-} \varphi .
\end{split}\end{align}
Substituting (\ref{SteadyNonConsStep5}) into (\ref{SteadyNonConsStep3}) then gives:
\begin{align}\begin{split}\label{SteadyNonConsStep6}\MoveEqLeft[1]
\int_{\Omega_-} \beta_+ \nabla\left(\tu-u\right)\cdot \nabla \varphi \\
&{}= -\int_{\Omega_-} \nabla \beta_+ \cdot \nabla\tu \, \varphi +\int_{\Omega_-} f\,\frac{\beta_+}{\beta_-} \varphi + \int_{\Omega_-} \frac{\beta_+}{\beta_-} \nabla \beta_-\cdot \nabla \tu\,\varphi + \int_{\Omega_-} \nabla\cdot \left( \beta_+ \nabla u \right) \varphi.
\end{split}\end{align}
The definition of (\ref{pb1NoAlgNonCons}) ensures that for all $\varphi$ in $H_0^1(\Omega_-)$,
\begin{align}\begin{split}\label{SteadyNonConsStep7}
-\int_{\Omega_-} \nabla\cdot \left( \beta_+ \nabla u \right) \varphi 
&= \int_{\Omega_-} \frac{\beta_+}{\beta_-} f\,\varphi + \int_{\Omega} \frac{\beta_+}{\beta_-} \nabla \beta_- \cdot \nabla \tu \, \varphi - \int_{\Omega_-} \nabla\beta_+\cdot\nabla\tu \, \varphi.
\end{split}\end{align}
Substituting (\ref{SteadyNonConsStep7}) into the last term on the right-hand side of (\ref{SteadyNonConsStep6}) establishes:
\begin{align}\label{SteadyNonConsStep8}
\int_{\Omega_-} \beta_+ \nabla\left(\tu-u\right)\cdot \nabla \varphi &= 0,
\end{align}
Since $u=\tu$ in $\Omega_-$, one may then easily verify that $u$ satisfies:
\begin{align}
	-\nabla \cdot \left(\beta \nabla u\right) &= f \qquad \text{in } \Omega\slash\gamma.
\end{align}
and that:
\begin{align}\label{jumps}
\jump{\beta_+ \nabla u \cdot \nn} = (\beta_+ - \beta_- ) \frac{\partial \widetilde u}{\partial \nn} = (\beta_+ - \beta_- ) \frac{\partial u}{\partial \nn},
\end{align}
implying that $\jump{\beta \nabla \cdot \nn}=0$, completing the proof.
\end{proof}
\subsection{Proof of Theorem \ref{thm:nonconsUnsteady}}
\begin{proof}
Let $u$ be a solution of (\ref{Model2}) and define $\tu$ as in (\ref{uTildeDefNL}). Then familiar arguments give:
 for $u \in H^{1}\left(\Omega\right)$ and $\varphi \in C_{0}^{\infty}\left(\Omega\right)$:
\begin{align}\begin{split}\label{proofStep2UnstNL}\MoveEqLeft[4]
\int_{\Omega} \rho \frac{\partial u}{\partial t}\varphi + \int_{\Omega} \beta \nabla{u} \cdot \nabla{\varphi} \\ \quad{}& = \int_{\Omega} \rho_+ \frac{\partial u}{\partial t} \varphi - \int_{\Omega_-} \rho_+ \frac{\partial u}{\partial t} \varphi + \int_{\Omega_-} \rho_- \frac{\partial u}{\partial t} \varphi \\
\quad{}& +  \int_{\Omega} \beta_+ \nabla{u} \cdot \nabla{\varphi} - \int_{\Omega_-} \beta_+ \nabla u \cdot \nabla \varphi + \int_{\Omega_-} \beta_- \nabla u \cdot \nabla \varphi \\
\quad{}& = \int_{\Omega} \rho_+ \frac{\partial u}{\partial t} \varphi - \int_{\Omega_-} \rho_+ \frac{\partial \widetilde{u}}{\partial t} \varphi + \int_{\Omega_-} \rho_- \frac{\partial \widetilde{u}}{\partial t} \varphi \\
\quad{}& +  \int_{\Omega} \beta_+ \nabla{u} \cdot \nabla{\varphi} - \int_{\Omega_-} \beta_+ \nabla \widetilde{u} \cdot \nabla \varphi + \int_{\Omega_-} \beta_- \nabla \widetilde{u} \cdot \nabla \varphi \\
\end{split}\end{align}
Since $u$ satisfies (\ref{Model2}), from (\ref{proofStep2UnstNL}):
\begin{align}\begin{split}\label{moreClarity2NL}\MoveEqLeft[4]	
	\int_{\Omega} \rho_+ \frac{\partial u}{\partial t} \varphi - \int_{\Omega_-} \rho_+ \frac{\partial \widetilde{u}}{\partial t} \varphi + \int_{\Omega_-} \rho_- \frac{\partial \widetilde{u}}{\partial t} \varphi \\
\quad{}& +  \int_{\Omega} \beta_+ \nabla{u} \cdot \nabla{\varphi} - \int_{\Omega_-} \beta_+ \nabla \widetilde{u} \cdot \nabla \varphi + \int_{\Omega_-} \beta_- \nabla \widetilde{u} \cdot \nabla \varphi \\
\quad{}&= \int_{\Omega} f \varphi  \text{  in  } \Omega,
\end{split}\end{align}
and in particular:
\begin{align}\label{proofStep3NLU}
-\nabla \cdot \left(\beta_-\nabla\tu\right)&= f - \rho_-\frac{\partial \tu}{\partial t} \,\,\,\, \text{     in     } \Omega_-.
\end{align}
After applying (\ref{divergenceProductRule}) to (\ref{proofStep3NLU}), multiply by $\beta_+$ and rearranging:
\begin{align}\label{proofStep3NLU}
-\beta_+ \Delta \tu  &= \frac{\beta_+}{\beta_-} f - \frac{\beta_+\rho_-}{\beta_-} \frac{\partial \tu}{\partial t} + \frac{\beta_+}{\beta_-}\nabla \tu \cdot \nabla \beta_- \,\,\,\, \text{     in     } \Omega_-.
\end{align}
As done in (\ref{lemmaStep2}), we add and subtract $\nabla\tu\cdot\nabla\beta_+$ and apply (\ref{divergenceProductRule}), yielding:
\begin{align}\label{proofStep4FINAL}
- \nabla \cdot \left( \beta_+ \nabla \tu\right)  &= \frac{\beta_+}{\beta_-} f - \frac{\rho_-\beta_+}{\beta_-} \frac{\partial \tu}{\partial t} + \frac{\beta_+}{\beta_-}\nabla \tu \cdot \nabla \beta_- - \nabla\tu\cdot\nabla\beta_+ \,\,\,\, \text{     in     } \Omega_-.
\end{align}
Following the same approach as in the previous results, one integrates by parts the terms:
\begin{align}\begin{split}\label{proofStep4NLU}\MoveEqLeft[1] 
- \int_{\Omega_-} \beta_+ \nabla \widetilde{u} \cdot \nabla \varphi + \int_{\Omega_-} \beta_- \nabla \widetilde{u} \cdot \nabla \varphi \\ \quad{}& = \int_{\Omega_-} \nabla \cdot\left( \beta_+ \nabla \widetilde{u}\right) \varphi - \int_{\gamma} \beta_{+} \frac{\partial \widetilde{u}}{\partial \nn} \varphi - \int_{\Omega_-} \nabla \cdot\left( \beta_- \nabla \widetilde{u}\right)  \varphi + \int_{\gamma} \beta_{-} \frac{\partial \widetilde{u}}{\partial \nn} \varphi. \\
\end{split}\end{align}	
Substituting (\ref{proofStep3NLU}) and (\ref{proofStep4FINAL}) into (\ref{proofStep4NLU}) above:
\begin{align}\begin{split}\label{proofStep5NLU}\MoveEqLeft[1] 
- \int_{\Omega_-} \beta_+ \nabla \widetilde{u} \cdot \nabla \varphi
+ \int_{\Omega_-} \beta_- \nabla \widetilde{u} \cdot \nabla \varphi \\ \quad{}& =  \int_{\Omega_-} \left(1-\frac{\beta_+}{\beta_-}\right) f \varphi 
 +\int_{\Omega_-} \left(\frac{\rho_-\beta_+}{\beta_-}-\rho_-\right) \frac{\partial \tu}{\partial t}
  \\ \quad{}& - \int_{\gamma} \beta_{+} \frac{\partial \widetilde{u}}{\partial \nn} \varphi  
+\int_{\Omega_-} \nabla \tu \cdot \nabla \beta_+ - \int_{\Omega_-} \frac{\beta_+}{\beta_-}\nabla \tu \cdot \nabla \beta_-
  + \int_{\gamma} \beta_{-} \frac{\partial \widetilde{u}}{\partial \nn} \varphi. \\
\end{split}\end{align}
Combining (\ref{proofStep2NL}) and (\ref{proofStep5NLU}):
\begin{align}\begin{split}\label{proofStepENDNLUNST}\MoveEqLeft[1]
\int_{\Omega} \rho_+ \frac{\partial u}{\partial t} \varphi + \int_{\Omega} \beta_+ \nabla u \cdot \nabla \varphi \\ \quad{}& = 
\int_{\Omega_-} \left(\rho_+ - \frac{\rho_- \beta_+}{\beta_-} \right)\frac{\partial \widetilde{u} }{\partial t} \varphi + \int_{\Omega_+} f\varphi + \int_{\Omega_-} \frac{\beta_+}{\beta_-} f\varphi  \\ \quad{}& + \int_{\gamma}\beta_+ \frac{\partial \widetilde{u} }{\partial \nn} \varphi -  \int_{\gamma} \beta_-\frac{\partial \widetilde{u} }{\partial \nn} \varphi -\int_{\Omega_-} \nabla \tu \cdot \nabla \beta_+ + \int_{\Omega_-} \frac{\beta_+}{\beta_-}\nabla \tu \cdot \nabla \beta_-,
\end{split}\end{align}
which was to be shown. 
\par To prove the other direction, we proceed as in the previous analyses: let $(u,\,\tu)$ solve (\ref{pb1NoAlgNonConsUnst}) and (\ref{pb2NoAlgNonConsUnst}) and show that $\tu-u$ solves (\ref{verifPbmK}), after which backward integration over (\ref{proofStepENDNLUNST}) gives the result. 
\par For all $\varphi \in H_0^1(\Omega_-(t))$:
\begin{align}\label{stp1}
\int_{\Omega_-} \rho_+ \frac{\partial (\tu-u)}{\partial t}\varphi - \int_{\Omega_-} \nabla\cdot\left\lbrack \beta_+ \nabla(\tu-u)\right\rbrack\varphi &= 0
\end{align}
 and hence:
 \begin{align}\label{stp2}
\int_{\Omega_-} \rho_+ \frac{\partial \tu}{\partial t}\varphi - \int_{\Omega_-} \nabla\cdot\left( \beta_+ \nabla\tu\right)\varphi &= \int_{\Omega_-} \rho_+ \frac{\partial u}{\partial t}\varphi - \int_{\Omega_-} \nabla\cdot\left( \beta_+ \nabla u \right)\varphi
\end{align}
Applying (\ref{proofStep4FINAL}) above:
\begin{align}\begin{split}\label{stp3}\MoveEqLeft[1]
\int_{\Omega_-} \left(\rho_+ -\frac{\beta_+\rho_-}{\beta_-}\right) \frac{\partial \tu}{\partial t}\varphi + \int_{\Omega_-} \frac{\beta_+}{\beta_-} f\varphi - \int_{\Omega_-}\frac{\beta_+}{\beta_-}\nabla\tu\cdot\nabla\beta_-\varphi +\int_{\Omega_-} \nabla \tu \cdot \nabla\beta_+ \varphi \\ 
{}&= \int_{\Omega_-} \rho_+ \frac{\partial u}{\partial t}\varphi - \int_{\Omega_-} \nabla\cdot\left( \beta_+ \nabla u \right)\varphi.
\end{split}\end{align}
From the definition of Problem (\ref{pb1NoAlgNonConsUnst}), the left and right hand sides are equal, establishing $\tu=u$ on $\Omega_-$. We note that jump condition is satisfied for identical reasons as before, completing the proof.
\end{proof}

\bibliography{Viguerie_Veneziani.bib}

\begin{thebibliography}{10}

\bibitem{BMM2005}
S.~Bertoluzza, M.~Ismail, and B.~Maury.
\newblock The {F}at boundary method: Semi-discrete scheme and some numerical
  experiments.
\newblock In {\em Domain Decomposition Methods in Science and Engineering},
  volume~40 of {\em Lecture Notes in Computational Science and Engineering},
  pages 513--520. Springer Berlin Heidelberg, 2005.

\bibitem{BIM2011}
S.~Bertoluzza, M.~Ismail, and B.~Maury.
\newblock Analysis of the fully discrete fat boundary method.
\newblock {\em Numerische Mathematik}, 118:49--77, 2011.

\bibitem{BRDP2018}
C.~Bruna-Russo, A.~G{\"o}khan~Demir, and B.~Previtali.
\newblock Selective laser melting finite element modeling: {V}alidation with
  high--speed imaging and lack of fusion defects prediction.
\newblock {\em Materials and Design}, 156:143--153, 2018.

\bibitem{CBB2013}
F.~Craveiro, H.~Bartolo, and P.~Bartolo.
\newblock Functionally graded structures through building manufacturing.
\newblock {\em Advanced Materials Research}, 683:775--778, 2013.

\bibitem{DBRR2002}
K.~Davey, S.~Bounds, I.~Rosindale, and M.T. Rasgado.
\newblock A coarse preconditioner for multi-domain boundary element equations.
\newblock {\em Computers and Structures}, 80(7--8):643--658, 2002.

\bibitem{DWZ2009}
Q.~Du, D.~Wang, and L.~Zhu.
\newblock On mesh geometry and stiffness matrix conditioning for general finite
  element spaces.
\newblock {\em SIAM J. Numer. Anal.}, 47(2):1421--1444, 2009.

\bibitem{Elman}
Howard~C Elman, David~J Silvester, and Andrew~J Wathen.
\newblock {\em Finite elements and fast iterative solvers: with applications in
  incompressible fluid dynamics}.
\newblock Oxford University Press (UK), 2014.

\bibitem{EO1992}
F.~Erdogan and M.~Ozturk.
\newblock Diffusion problems in bonded nonhomogenous materials with an
  interface cut.
\newblock {\em International Journal of Engineering Science},
  30(10):1507--1523, 1992.

\bibitem{GGZ2007}
X.~Gao, L.~Guo, and C.~Zhang.
\newblock Three-step multi-domain {BEM} solver for nonhomogenous material
  problems.
\newblock {\em Engineering Analysis with Boundary Elements}, 31(12):965--973,
  2007.

\bibitem{GMWP2012}
D.D. Gu, W.~Meiners, K.~Wissenbach, and R.~Poprawe.
\newblock Laser additive manufacturing of metallic components: materials,
  processes, and mechanisms.
\newblock {\em International Materials Reviews}, 57(3):133--164, 2012.

\bibitem{HYDY2016}
Y.~Huang, L.J. Yang, X.Z. Du, and Y.P. Yang.
\newblock Finite element analysis of thermal behavior of metal powder during
  selective laser melting.
\newblock {\em International Journal of Thermal Sciences}, 104:146--157, 2016.

\bibitem{IM2016}
J.~Irwin and P.~Michaleris.
\newblock A line heat input model for additive manufacturing.
\newblock {\em J. Manuf. Sci. Eng.}, 138(11), 2016.

\bibitem{KHX2014}
L.~Kameski, W.~Huang, and H.~Xu.
\newblock Conditioning of finite element equations with arbitrary anisotropic
  meshes.
\newblock {\em Math. Comput.}, 83:2187--2211, 2014.

\bibitem{KAFHKKR2015}
W.~E. King, A.~T. Anderson, R.~M. Ferencz, N.~E. Hodge, C.~Kamath, S.~A.
  Khairallah, and A.~M. Rubenchik.
\newblock Laser powder bed fusion additive manufacturing of metals; physics,
  computational, and materials challenges.
\newblock {\em Applied Physics Reviews}, 2(4), 2015.

\bibitem{KOCZR2018}
S.~Kollmannsberger, A.~Ozcan, M.~Carraturo, N.~Zander, and E.~Rank.
\newblock A hierarchical computational model for moving thermal loads and phase
  changes with applications to selective laser melting.
\newblock {\em Computers and Mathematics with Applications}, 75(5):1483--1497,
  2018.

\bibitem{LRMR2016}
N.~Labonnote, A.~Ronnquist, B.~Manum, and P.~Ruther.
\newblock Additive construction: {S}tate-of-the-art, challenges and
  opportunities.
\newblock {\em Automation in Construction}, 72(3):347--366, 2016.

\bibitem{LWPLK2009}
D.V. Le, J.~White, J.~Peraire, K.M. Lim, and B.C. Khoo.
\newblock An implicit immersed boundary method for three-dimensinal
  fluid-membrane interactions.
\newblock {\em Journal of Computational Physics}, 228:8427--8445, 2009.

\bibitem{LL2003}
L.~Lee and R.J. Leveque.
\newblock An immersed interface method for incompressible {N}avier-{S}tokes
  equations.
\newblock {\em SIAM J. Sci. Comput.}, 25(3):832--856, 2003.

\bibitem{LL1997}
R.J. Leveque and Z.~Li.
\newblock Immersed interface methods for {S}tokes flow with elastic boundaries
  or surface tension.
\newblock {\em SIAM J. Sci. Comput.}, 18(3):709--735, 1997.

\bibitem{Maury2001}
B.~Maury.
\newblock A {F}at boundary method for the {P}oisson problem in a domain with
  holes.
\newblock {\em J. Sci. Comp.}, 16(3):319--339, 2001.

\bibitem{PPRP2005}
J.R. Pacheco, A.~Pacheco-Vega, T.~Rodi, and R.E. Peck.
\newblock Numerical simulations of heat transfer and fluid flow problems using
  an immersed-boundary finite-volume method on non-staggered grids.
\newblock {\em Numer. Heat Tr. B-Fund.}, 48:1--24, 2005.

\bibitem{PPRZMHBS2015}
N.~Patil, D.~Pal, H.~Khalid~Rafi, K.~Zeng, A.~Moreland, A.~Hicks, D.~Beeler,
  and B.~Stucker.
\newblock A generalized feed forward dynamic adaptive mesh refinement and
  derefinement finite element framework for metal laser sintering--{P}art {I}:
  Formulation and algorithm development.
\newblock {\em Journal of Manufacturing Science and Engineering}, 137(4), 2015.

\bibitem{PPRZMHBS2016}
N.~Patil, D.~Pal, K.H. Kutty, K.~Zeng, A.~Moreland, A.~Hicks, D.~Beeler, and
  B.~Stucker.
\newblock A generalized feed forward dynamic adaptive mesh refinement and
  derefinement finite element framework for metal laser sintering--{P}art {II}:
  Nonlinear thermal simulations and validations.
\newblock {\em Journal of Manufacturing Science and Engineering}, 138(6), 2016.

\bibitem{RSM2014}
D.~Riedlbauer, P.~Steinmann, and J.~Mergheim.
\newblock Thermomechanical finite element simulations of selective electron
  beam melting processes: Performance considerations.
\newblock {\em Comput. Mech.}, 54(1):109--122, 2014.

\bibitem{RSZ2018}
M.A. Russell, A.~Suoto-Iglesias, and T.I. Zohdi.
\newblock Numerical simualtion of {L}aser {F}usion {A}dditive {M}anufacturing
  processes using the {S}{P}{H} method.
\newblock {\em Computer Methods in Applied Mechanics and Engineering},
  341:163--187, 2018.

\bibitem{TT2017}
S.~Tammas-Williams and I.~Todd.
\newblock Design for additive manufacturing with site-specific properties in
  metals and alloys.
\newblock {\em Scripta Materiala}, 135:105--110, 2017.

\bibitem{WLGNMR2017}
Q.~Wang, J.~Li, M.~Gouge, A.R. Nassar, P.P. Michaleris, and E.W. Reutzel.
\newblock Physics-based multivariable modeling and feedback linearization
  control of melt-pool geometry and temperature in directed energy deposition.
\newblock {\em J. Manuf. Sci. Eng.}, 139(2), 2017.

\end{thebibliography}
\bibliographystyle{plain}

\end{document}